\def\File-version{1.0b}    % version of alggeom.cls documented

\documentclass[noams]{alggeom}

\usepackage{amsmath, amssymb}
\usepackage{amsfonts, epsfig,color,wrapfig}
\usepackage{ xypic,verbatim,amscd,color,hyperref}
\usepackage{graphicx}

\renewcommand\theenumii{\roman{enumii}}
\newcommand\lL{\mathfrak{\theta}}
\newcommand\GLL{\operatorname{GL}}
\newcommand\mnn{\mathcal{N}}

\newcommand\bull{\sssize{\bullet}}

\newcommand\tr{\tilde{r}}

\newcommand\mvv\Bbb

\newcommand\frg{\mathfrak{g}}

\newcommand\mm{\mathcal{M}}

\newcommand\mf{\mathcal{F}}

\newcommand\pone{\Bbb{P}^1}

\newcommand\me{\mathcal{E}}
\newcommand\eva{\operatorname{ev}}
\newcommand\tmw{\overline{\mathcal{W}}}

\newcommand\mt{\mathcal{T}}

\newcommand\Fl{\operatorname{Fl}}

\newcommand\mv{\mathcal{V}}
\newcommand\tensor{\otimes}
\newcommand\ml{\mathcal{L}}

\newcommand\ma{\mathcal{A}}

\newcommand\mb{\mathcal{B}}

\newcommand\GL{\operatorname{GL}}

\newcommand\mg{\frg}

\newcommand\im{\operatorname{im}}
\newcommand\mw{\mathcal{W}}
\newcommand\mc{\mathcal{C}}
\newcommand\mq{\mathcal{Q}}
\newcommand\rk{\operatorname{rk}}
\newcommand\ms{\mathcal{S}}

\newcommand\pr{\operatorname{pr}}

\newcommand\xpp{x}

\newcommand\Hom{\operatorname{Hom}}

\newcommand\Gr{\operatorname{Gr}}

\newcommand{\leto}[1]{\stackrel{#1}{\to}}

\newcommand\mgg{\mathcal{G}}

\newcommand\tmvv{\overline{\mathcal{V}}}

\newcommand\SLL{\operatorname{SL}}
\newcommand\mfm{\operatorname{Gr}}

\newcommand\mpp{{\operatorname{ S}}}

\newcommand\quot{\operatorname{Quot}}
\newcommand\vr{\Bbb{V}_{\sL_{r+1},\vec{\lambda},\ell}}
\newcommand\vl{\Bbb{V}_{\sL_{\ell+1},\vec{\lambda}^T,r}}

\newcommand\vrplusone{\Bbb{V}_{\sL_{r+1},\vec{\lambda},\ell+1}}
\newcommand\vlplusone{\Bbb{V}_{\sL_{\ell+1},\vec{\lambda}^T,r+1}}

\newcommand\clar{\Bbb{A}_{\sL_{r+1},\vec{\lambda}}}

\newcommand\clal{\Bbb{A}_{\sL_{\ell+1},\vec{\lambda}^T}}

\newcommand\aclar{{A}_{\sL_{r+1},\vec{\lambda}}}
\newcommand\aclal{{A}_{\sL_{\ell+1},\vec{\lambda}^T}}

\newcommand\parbun{\mathcal{P}{ar}}
\newcommand{\ovop}[1]{\overline{\operatorname{#1}}}
\newcommand\cone{\alpha}
\newcommand\ctwo{\beta}
\newcommand\op{\operatorname}

\newtheorem{theorem}{Theorem}[section]

\newtheorem{remark}[theorem]{ Remark}
\newtheorem{conjecture}[theorem]{Conjecture}

\newtheorem{question}[theorem]{Question}
\newtheorem{proposition}[theorem]{Proposition}

\newtheorem{lemma}[theorem]{Lemma}

\newtheorem{definition}[theorem]{Definition}
\newtheorem{definition/lemma}[theorem]{Definition/Lemma}
\newtheorem{defi}[theorem]{Definition}
\newcommand{\sL}{\mathfrak{sl}}

\begin{document}

\title{Vanishing and identities of conformal blocks divisors }
\author{Prakash Belkale}
\email{belkale@email.unc.edu}
\address{Department of Mathematics, University of North Carolina, Chapel Hill, NC 27599}
\author{Angela Gibney}
\email{agibney@math.uga.edu}
\address{Department of Mathematics, University of Georgia, Athens, GA 30602}

\author{Swarnava Mukhopadhyay}
\email{swarnava@umd.edu}
\address{Department of Mathematics, University of Maryland,
College Park, MD 20742}
%
%\dedication{A dedication can be included here.}

\thanks{P.B. and S.M.  were supported on NSF grant  DMS-0901249, and A.G. on NSF grant DMS-1201268.}

\begin{abstract}
Conformal block divisors in type A  on $\ovop{M}_{0,\op{n}}$ are shown to satisfy new symmetries when levels and ranks
are interchanged in non-standard ways.   A connection with the quantum cohomology of Grassmannians reveals that these divisors vanish above the critical level.
\end{abstract}

\maketitle

%\vspace*{6pt}\tableofcontents  % for this guide only.
% A table of contents should normally not be included

\section{Introduction}
Given a simple Lie algebra $\mg$ and suitable choice of weights $\vec{\lambda}$ for $\mg$ at level $\ell$, there is a globally generated vector bundle $\mathbb{V}_{\mg,\vec{\lambda},\ell}$ of conformal blocks on the moduli space $\ovop{M}_{0,\op{n}}$ of stable, n-pointed rational curves.  Global generation on $\ovop{M}_{0,\op{n}}$ comes from the surjection
$$\mathbb{A}_{\mg, \vec{\lambda}} =\op{A}_{\mg, \vec{\lambda}}  \times  \ \ovop{M}_{0,\op{n}}\twoheadrightarrow \mathbb{V}_{\mg,\vec{\lambda},\ell},  \ \,  \ \mbox{ where }    \op{A}_{\mg, \vec{\lambda}}=(\tensor_{i=1}^n V_{\lambda_i})_{\frg}=\frac{\tensor_{i=1}^n V_{\lambda_i}}{\frg (\tensor_{i=1}^n V_{\lambda_i})}$$
is the vector space of coinvariants, the largest quotient space on which $\mg$ acts trivially.  Here $V_{\lambda}$ is the irreducible finite dimensional representation of $\frg$ with highest weight $\lambda$.  Consequently,  one obtains a morphism $f_{\mathbb{V}}$ to a Grassmannian variety of $\op{rk}\mathbb{V}_{\mg,\vec{\lambda},\ell}$ quotients of $\op{A}_{\mg, \vec{\lambda}}$:
$$\overline{\operatorname{M}}_{0,\op{n}} \overset{f_{\mathbb{V}}}{\longrightarrow}  \operatorname{Grass}^{quo}(\op{rk}\mathbb{V}_{\mg,\vec{\lambda},\ell},\op{A}_{\mg, \vec{\lambda}} ) \overset{p}{\hookrightarrow} \mathbb{P}^{N-1},  \ \mbox{ where } {N={{\op{rk}\mathbb{A}_{\mg, \vec{\lambda}}}\choose{{\op{rk}\mathbb{V}_{\mg,\vec{\lambda}}}}}}.$$

  The composition of  $f_{\mathbb{V}}$ with the Pl\"ucker embedding $p$, is given by  $c_1(\mathbb{V}_{\frg,\vec{\lambda},\ell})$.  These so-called conformal blocks divisors  $\mathbb{D}_{\mg,\vec{\lambda},\ell}=c_1 (\mathbb{V}_{\mg,\vec{\lambda},\ell})$ generate a full dimensional sub-cone in the nef cone.

 It is well-known that given   $\mg$ and  $\vec{\lambda}$, there is an integer $\ell_0(\mg, \vec{\lambda})$, for which
  $\mathbb{V}_{\mg,\vec{\lambda},\ell}$ and $\mathbb{A}_{\mg, \vec{\lambda}}$ are isomorphic
 for all $\ell > \ell_0$.   In particular, $\operatorname{rk}\mathbb{A}_{\mg, \vec{\lambda}}=\operatorname{rk}\mathbb{V}_{\mg,\vec{\lambda},\ell}$, and so the conformal blocks divisor $\mathbb{D}_{\mg,\vec{\lambda},\ell}$ vanishes, and its map contracts everything, for all $\ell > \ell_0$.

Here we study two different bounds for $\ell_0$:  the {\em{critical level}} $c(\sL_{r+1},\vec{\lambda})$, and  the {\em{theta level}} $\theta(\mg,\vec{\lambda})$.    The critical level, which we introduce, is related to an interpretation of the ranks of the bundles with $\mg=\sL_{r+1}$ to the quantum cohomology of the Grassmannian.  If $\ell$ is greater than either bound, then $\mathbb{D}_{\mg,\vec{\lambda},\ell}=0$ (Theorem \ref{newthing} and Remark \ref{ThetaVanishing}).
While equal for $r=1$, the critical and theta levels are generally distinct, reflecting different aspects of the weights $\vec{\lambda}$.

As an application of vanishing, we give sufficient conditions for divisors $\mathbb{D}_{\frg,\vec{\lambda},\ell}$ to be extremal in the nef cone (Propositions \ref{extremal1} and \ref{extremal2}), and show the morphisms they define factor through birational contraction maps to Hassett spaces (Theorems \ref{Smallerd} and \ref{Smallerdnew}).   We find a surprising relationship between pairs of
critical level bundles  $\mathbb{V}_{\sL_{r+1},\vec{\lambda},\ell}$ and $\mathbb{V}_{\sL_{\ell+1},\vec{\lambda}^T,r}$, where $\ell$ is the critical level for $(\sL_{r+1},\vec{\lambda})$.  In this case, $r$ is the critical level for the pair $(\sL_{\ell+1},\vec{\lambda}^T)$, where $\vec{\lambda}^T=(\lambda_1^T,\ldots,\lambda_n^T)$ are the representations whose Young diagrams $\lambda_i^T$ are given by the transposed Young diagrams associated to the $\lambda_i$ as described in Section \ref{Notation}.  Namely, while the bundles $\mathbb{V}_{\sL_{r+1},\vec{\lambda},\ell}$ and $\mathbb{V}_{\sL_{\ell+1},\vec{\lambda}^T,r}$ are not isomorphic, their first Chern classes are equal, and the images of the maps they define are related by Grassmann duality (Proposition \ref{strange}).  

In addition to the identity $\mathbb{D}_{\sL_{r+1},\vec{\lambda},\ell}=\mathbb{D}_{\sL_{\ell+1},\vec{\lambda}^T,r}$ for critical level pairs, we show that $\mathbb{D}_{\sL_{r+1},\vec{\lambda},\ell}=\mathbb{D}_{\sL_{r+1},\vec{\lambda}^*,\ell}$ (Proposition \ref{relation}).  Here $\vec{\lambda}^*=(\lambda_1^{*},\ldots,\lambda_n^{*})$, and $\lambda^*$ denotes the highest weight of $(V_{\lambda})^*$.
Taken together, these identities enhance our criteria for detecting vanishing and understanding of the nef cone.

\subsection{Definitions, and precise statement of results}
We now state our results, outline the paper and our methods, and put the work into context.

\subsubsection{Notation}\label{Notation}
For a finite dimensional simple Lie algebra $\frg$, and a positive integer  $\ell$ (called the level), let ${P}_{\ell}(\frg)$ denote the set of dominant integral weights $\lambda$ with 
$(\lambda,\theta)\leq \ell$.  Here $\theta$ is the highest root, and $(\ ,\ )$ is the Killing form, normalized so that $(\theta,\theta)=2$.  If $H_{\theta}$ is the co-root corresponding to the highest root $\theta$, then $(\lambda,\theta)=\lambda(H_{\theta})$.

To a triple $(\frg,\vec{\lambda},\ell)$, such that $\vec{\lambda} \in P_{\ell}(\frg)^n$, there corresponds a  a vector bundle of conformal blocks $\mathbb{V}_{\frg,\vec{\lambda},\ell}$ on the moduli stack $\overline{\mathcal{M}}_{g,n}$ \cite{TUY} (see also \cite{Tsuchimoto,Sorger,Fakh}).

Finite dimensional irreducible polynomial representations for $\operatorname{GL}_{r+1}$ are parameterized by Young diagrams $\lambda=(\lambda^{(1)}\geq \lambda^{(2)}\geq\dots\geq\lambda^{(r)}\geq \lambda^{(r+1)}\geq 0)$.  Young diagrams $\lambda$ and $\mu$ give the same representation of $\operatorname{SL}_{r+1}$ (equivalently $\sL_{r+1}$) if $\lambda^{(a)}-\mu^{(a)}$ is a constant independent of $a$.
We use the notation $|\lambda|=\sum_{i=1}^r \lambda^{(i)}$.  We say that $\lambda\in P_{\ell}(\sL_{r+1})$ if and only if $\lambda^{(1)}-\lambda^{(r+1)}\leq \ell$.  We refer to $\lambda$ as normalized if $\lambda^{(r+1)}=0$. The normalization of $\lambda$ is 
$\lambda-\lambda^{(r+1)}\cdot(1,1,\dots,1)$. Also note that $(\lambda,\theta)=\lambda(H_{\theta})=\lambda^{(1)}-\lambda^{(r+1)}$.

\subsubsection{The critical level}

\begin{defi} \label{CritLevDefi} Let $\vec{\lambda}=(\lambda_1,\dots,\lambda_n)$ be an $n$-tuple of  {\em normalized} integral weights for $\sL_{r+1}$, assume that $r+1$ divides $\sum_{i=1}^n|\lambda_i|$, and define the critical level for the pair  $(\sL_{r+1},\vec{\lambda})$ to be
$$ c(\sL_{r+1},\vec{\lambda})=-1+\frac{1}{r+1}\sum_{i=1}^n|\lambda_i|.$$
One can define $c(\sL_{r+1},\vec{\lambda})$ in general, by replacing each $\lambda_i$ by its normalization.
\end{defi}

\begin{remark}\label{remarkI}
\begin{enumerate}
\item The non-zeroness of $\rk \mathbb{A}_{\frg,\vec{\lambda}}$ (similarly $\rk\mathbb{V}_{\frg, \vec{\lambda},\ell}$) for $\frg=\sL_{r+1}$ is controlled by a non-trivial system of inequalities ~\cite{klyachko,knutsontao,fulton,b4,belkaleicm}. It is therefore unreasonable to look for an optimal critical level valid for all data.
\item The ranks of  $ \mathbb{A}_{\frg,\vec{\lambda}}$  and $\mathbb{V}_{\frg, \vec{\lambda},\ell}$ coincide with the ranks of 
global sections of line bundles over suitable moduli spaces - the moduli spaces of (semistable) filtered vector spaces, and
of semistable parabolic bundles respectively. One may look for levels at which parabolic semistable bundles are necessarily trivial (as bundles), so that these moduli spaces and line bundles are the same, and hence the ranks coincide. The resulting bounds for vanishing are weaker than the critical level bounds (by one), see Section ~\ref{Z} and Question \ref{GenCritLev}.
\end{enumerate}
\end{remark}

\begin{proposition}\label{newthing}
Suppose $\vec{\lambda} \in P_{\ell}(\sL_{r+1})^n$, and
$\ell>c(\sL_{r+1},\vec{\lambda})$. Then
$ \clar = \Bbb{V}_{\sL_{r+1},\vec{\lambda},\ell}$, and in particular, one has that $\Bbb{D}_{\sL_{r+1},\vec{\lambda},\ell}=c_1(\Bbb{V}_{\sL_{r+1},\vec{\lambda},\ell})=0.$
\end{proposition}

The proof of Proposition \ref{newthing}, given in Section \ref{newthingproof}, follows from an enumerative interpretation of conformal blocks \cite{witten,agni} and  of classical invariants for $\sL_{r+1}$ described in Section \ref{WittenEnumerative}.  As is explained in Section \ref{eleven} one can compute  the rank of $\Bbb{V}_{\sL_{r+1},\vec{\lambda},\ell}$ by calculating a particular (quantum) product of Schubert cycles in a Grassmannian.  In  case  $\ell>c(\sL_{r+1},\vec{\lambda})$, the corresponding quantum cohomology structure coefficient equals the classical structure constant, which gives the rank of the bundle of coinvariants.

The critical level is defined only for Lie algebras of form $\frg=\sL_{r+1}$. The theta level which we define next is valid for arbitrary simple Lie algebras.

\subsubsection{The theta level and comparison to the critical level}\label{CriticalLevelGeneral}
The theta level, defined below, arises from a description of $\mathbb{V}_{\frg,\vec{\lambda},\ell}$
as an explicit quotient of $\mathbb{A}_{\frg,\vec{\lambda}}$ \cite{GW,FSV1,FSV2,Beauville}. For $\frg=\sL_{r+1}$, this level can be related to  critical levels for related data (see Lemma \ref{averagenew}). Some basic properties of the theta level are explored in \cite{BGMB}.

\begin{defi}\label{FSVCriticalLevel}Given a pair $(\frg, \vec{\lambda})$, we refer to
$$\lL(\frg,\vec{\lambda})=  -1+ \frac{1}{2} \sum (\lambda_i,\theta)\in \frac{1}{2}\Bbb{Z}$$
as the theta level. 
\end{defi}

\begin{remark}\label{ThetaVanishing}\label{comparison}
\begin{enumerate}
\item It is an immediate consequence of the explicit quotient description of $\mathbb{V}_{\frg,\vec{\lambda},\ell}$, that if $\ell > \lL(\frg,\vec{\lambda})$, and $\vec{\lambda} \in P_{\ell}(\frg)$, then $\mathbb{V}_{\frg,\vec{\lambda},\ell}= \mathbb{A}_{\frg,\vec{\lambda}}$, and $\mathbb{D}_{\frg,\vec{\lambda},\ell}=c_1(\mathbb{V}_{\frg,\vec{\lambda},\ell})=0$ (see \cite{BGMB} and Lemma  \ref{averagenew}).
\item  It is easy to see that the theta and critical levels coincide for $\sL_2$. Lemma \ref{averagenew} (a) shows that the theta level for $(\sL_{r+1},\vec{\lambda})$ is the average
of the critical levels for $(\sL_{r+1},\vec{\lambda})$ and $(\sL_{r+1},\vec{\lambda}^*)$.
\end{enumerate}
\end{remark}

\subsubsection{Applications of vanishing}
A divisor on a projective variety $X$ is nef if it non-negatively intersects all curves on $X$.  The set of nef divisors forms  the nef cone, denoted $\op{Nef}(X)$.  Properties of $\op{Nef}(X)$ reflect aspects of the morphisms admitted by $X$.  Feasibility of studying the birational geometry of a particular space $X$  by way of its nef cone is highly variable.  On one extreme, if $X$ is a so-called Mori Dream Space, then $\op{Nef}(X)$ has a finitely number of extremal rays, and each extremal ray is spanned by a base point free divisor.

  It was recently  shown that  $\ovop{M}_{0,\op{n}}$ is not a Mori Dream Space for $n \ge 13$ \cite{MDSBS, GoKa}.  Whether there are an infinite number of distinct extremal rays of the nef cone
  and whether those rays are spanned by base point free divisors are two very open questions.

Propositions Props ~\ref{extremal1} and ~\ref{extremal2} give criteria for detecting when a conformal blocks divisor $\mathbb{D}$ lies on an extremal face of the nef cone given by a so-called   $\op{F}$-curve.  The extremality criteria allow us, in Theorems \ref{Smallerd} and \ref{Smallerdnew}, to prove that the morphisms given by certain conformal blocks divisors $\mathbb{D}_{\frg, \vec{\lambda},\ell}$  factor through maps to Hassett spaces $\overline{\operatorname{M}}_{0,\mathcal{A}}$, where the weight data $\mathcal{A}$ is determined by $\frg$, $\ell$ and
$\vec{\lambda}$.

  \subsubsection{Relations between divisors}

When finding where in the nef cone the conformal blocks divisors reside, we are interested not only in extremality (as results like Propositions ~\ref{extremal1} and ~\ref{extremal2} tell us about), but also in questions of independence.   An interesting feature of the set of conformal blocks divisors is that together they generate a dense sub-cone of the nef cone.  It is natural to wonder whether they cover the entire nef cone, and whether the cone they span is finitely generated.  To answer such questions we would like to know their dependence on one another.

We study two types of relations between conformal blocks divisors given by different sets of data.   First, in Proposition \ref{relation} we point out the natural identity  between divisors $c_1 (\mathbb{V}_{\frg,\vec{\lambda},\ell})=c_1 (\mathbb{V}_{\frg,\vec{\lambda}^*,\ell})$, that comes from an involution of the Weyl chamber.  Second, in Theorem \ref{main}, we identify pairs of so-called critical level partner divisors $c_1(\mathbb{V}_{\sL_{r+1},\vec{\lambda},\ell})=c_1(\mathbb{V}_{\sL_{\ell+1},\vec{\lambda}^T,r})$.    These divisor class identities multiply our extremality results.  Namely Proposition \ref{extremal1} holds for $\mathbb{D}_{\sL_{r+1}, \vec{\lambda}, \ell}$ and $\mathbb{D}_{\sL_{r+1}, \vec{\lambda}^*, \ell}$
(see Proposition \ref{relation});
  and Proposition \ref{extremal2} holds for each of $\mathbb{D}_{\sL_{r+1}, \vec{\lambda}, \ell}$  and
$\mathbb{D}_{\sL_{\ell+1}, \vec{\lambda}^T, r}$.

  We next discuss the critical level identities, which involve a more robust statement about the bundles themselves.
  To begin with, critical level bundles come in pairs.
If $\lambda_i \in P_{\ell}(\sL_{r+1})$, then $\lambda_i^T \in P_{r}(\sL_{\ell+1})$, where $\lambda_i^T$ is obtained by taking the transpose of the Young diagram associated to the weight $\lambda
  _i$.  Since $|\lambda_i|=|\lambda_i^T|$, it follow that
$\ell=c(\sL_{r+1},\vec{\lambda})$ if and only if $r=c(\sL_{\ell+1},\vec{\lambda^T})$.

\begin{proposition}\label{strange}
Suppose $\vec{\lambda} \in P_{\ell}(\sL_{r+1})^n$ where
$\ell= c(\sL_{r+1},\vec{\lambda})$. Then:
\begin{enumerate}
\item[(a)] one has that:
\begin{equation}\label{morpheusnew}
\rk \vr +\rk\vl= \rk \mathbb{A}_{\sL_{r+1}, \vec{\lambda}} =\rk \mathbb{A}_{\sL_{\ell+1}, \vec{\lambda}^T} ;
\end{equation}
\item[(b)] critical level partner divisors are the same: $$\mathbb{D}_{\sL_{r+1},\vec{\lambda},\ell}=\mathbb{D}_{\sL_{\ell+1},\vec{\lambda^T},r}.$$
\end{enumerate}
\end{proposition}

Definition \ref{CritLevDefi}, Theorem \ref{newthing}, and Proposition \ref{strange} (b) were discovered by Fakhruddin  for $r=1$ \cite{Fakh}.  He defines the critical level in Section $4.3$ and in Remark $5.3$ he compares the  divisors $$c_1(\mathbb{V}(\sL_{2},(i_1\omega_1,\ldots,i_n\omega_1),\ell)) \ \ \mbox{and } \ c_1(\mathbb{V}(\sL_{\ell+1},(\omega_{i_1},\ldots,\omega_{i_n}),1)),$$
satisfying $\sum_{j=1}^ni_j=2(\ell+1)$, pointing out that the formulas he derives in Proposition $5.2$ and Equation $4.3$ establish that the divisors are the same.

\begin{remark}
Our critical level symmetries are different from, but related to, the strange dualities considered in the literature  (e.g., \cite{Belkale1,MO2,NakTsu,RO}). See Section \ref{strangeandcritical}.
\end{remark}

Proposition \ref{strange} should be thought of as a corollary of the following result.

\begin{theorem}\label{main}
Suppose $\vec{\lambda} \in P_{\ell}(\sL_{r+1})^n$ where
$\ell= c(\sL_{r+1},\vec{\lambda})$. Then, there is a natural isomorphism:
\begin{equation}\label{morpheus}
\clar^*/\vr^*\leto{\sim} \vl.
\end{equation}
\end{theorem}

 \begin{remark}
 Proposition \ref{strange} part (a), and Theorem \ref{main}  hold even on $\ovop{M}_{0,3}$,  where the bundles are just vector spaces, giving new statements about their ranks.
\end{remark}

To prove Proposition \ref{strange} and Theorem \ref{main}, structure constants in the cohomology of a Grassmannian are decomposed into a sum of two quantum cohomology
structure constants  for different Grassmannians, yielding \eqref{morpheusnew}  (see Section ~\ref{leo}).

\subsection{A note on our methods} \medskip

The main results of this paper are proved by using relations of conformal blocks to generalized theta functions,
by arguments used in the geometric strange duality theory of vector bundles \cite{Belkale1, MO2, RO}, and in the study of quantum cohomology of Grassmannians (\cite{witten}, also see \cite{belkaleicm}). The applications use standard intersection theoretic computations on $\ovop{M}_{0,\op{n}}$ and the factorization formulas of \cite{TUY}. We recommend the Bourbaki article of Sorger ~\cite{Sorger} for some of the background on conformal blocks.

 In \cite{Fakh}, Fakhruddin gives explicit formulas for the classes $c_1 (\mathbb{V}_{\mg,\vec{\lambda},\ell})$. Explicit formulas for the ranks of
 $\vr$ are also known (the Verlinde formula, see Th{\'e}or{\`e}me 4.2.2 in \cite{Sorger}).  Therefore it is an interesting question to ask if Propositions \ref{newthing} and \ref{strange} can be proved using these explicit formulas. It is not clear however how to obtain these results from formulas (see Remark \ref{comparison_new}) because of difficulties with factorization data and ranks.

  \subsection{Acknowledgements}
We would like to thank N. Fakhruddin, D. Jensen, L. Mihalcea, and  H-B. Moon for useful discussions.   We are grateful to the referees for their comments and suggestions. Proposition \ref{newthing}  and statement (b) of Proposition \ref{strange} are a generalization and sharpening of conjectures made by the second author, and formed the starting point for this work.

\section{Schubert varieties and the cohomology of Grassmannians}\label{added}

\subsection{Notation}\label{notate}
We write $[m]=\{1,\dots,m\}$ for all positive integers $m$.  For an $m$-dimensional vector space  $W$, denote by  $\Fl(W)$ the space of complete flags of vector subspaces of $W$:
$$F_{\bull}:0\subsetneq E_1\subsetneq E_2\subsetneq \dots \subsetneq E_m=W.$$
The determinant line $\wedge^m W$ is denoted by $\det W$.  We fix a collection of $n$ distinct and ordered points $S=\{p_1,\dots,p_n\}\subseteq\pone$, and
for a vector bundle $\mw$ on $\pone$, define $\Fl_S(\mw)=\prod_{p\in S}\Fl(\mw_p).$
If $\me\in \Fl_S(\mw)$, we will assume that it is written in the form
$\me=(E^p_{\bull}\mid p\in S)$.

\subsection{Schubert varieties, and their tangent spaces}\label{twotwoone}
Let
$E_{\bull}\in \Fl(W)$ be a complete flag in an $M=r+1+k$-dimensional vector space $W$. Here we think of  $k$ as related to the level $\ell$.  Suppose $\lambda$ is a Young diagram that fits into a
$(r+1) \times k$ box. It is useful to associate to $\lambda$ an $(r+1)$-element subset $I$ of $[M]$ by
$I=\{i_1<\dots<i_{r+1}\},\  i_a=k+a -\lambda^{(a)},\  i_0=0,\ i_{(r+1)+1}= M$.
The open Schubert cell
$\Omega^o_I(E_{\bull})=\Omega^o_{\lambda}(E_{\sssize{\bullet}})$ is defined to be the smooth
subvariety of the Grassmannian $\Gr(r+1,W)$ of $(r+1)$-dimensional vector
subspaces of $W$ given by
$$ \Omega^o_{\lambda}(E_{\sssize{\bullet}})\ =\ \{ V \in \Gr((r+1),W) \mid \rk(V \cap
E_{j}) = a\text{ for } i_a\leq j< i_{a+1}, a\in[r+1]\}.$$

The closure $\Omega^o_{\lambda}(E_{\sssize{\bullet}})$  is denoted by $\Omega_{\lambda}(E_{\sssize{\bullet}})$. This is the set of $V\in\Gr(r+1,W)$ such that
$\rk(V\cap E_{i_a})\geq a,\  a\in[r+1]$.
The codimension of
$\Omega_{\lambda}(E_{\bull})$  in $\Gr(r+1,W)$ is $|\lambda|$.

Suppose $V\in \Omega^o_{\lambda}(E_{\sssize{\bullet}})$. Let $Q=W/V$. It is easy to see that  $V$ and $Q$ each receive induced flags from $E_{\bull}$; denote these by $F_{\bull}$
and $G_{\bull}$.
Then $T\Omega^o_{\lambda}(E_{\sssize{\bullet}})_V\subseteq T\Gr(r+1,W)_V=\operatorname{Hom}(V,Q)$ is
\begin{equation}\label{added2}T\Omega^o_{\lambda}(E_{\sssize{\bullet}})_V= \{\phi\in \operatorname{Hom}(V,Q)\mid \phi(F_a)\subseteq G_{k-\lambda^{(a)}}, a\in[r+1]\}.
\end{equation}

\subsection{Grassmann duality}\label{grdual} Let $W$ be an $M=r+1 +k$ dimensional vector space.  A natural identification $\Gr(r+1,W)=\Gr(k,W^*)$ takes a subspace $V\subseteq W$ to the kernel of the surjective map $W^*\to V^*$.  There is similarly an identification of the complete flag varieties $\Fl(W)$ and $\Fl(W^*)$. A Schubert variety $\Omega_{\lambda}(E_{\bull})\subseteq\Gr(r+1,W)$ can be identified with $\Omega_{\lambda^T}(E'_{\bull})\subseteq\Gr(k,W^*)$, where $E_{\bull}$ and $E'_{\bull}$ correspond under the identification of $\Fl(W)$ and $\Fl(W^*)$. A natural group isomorphism $\GLL(W)\to\GLL(W^*)$ acts equivariantly on the above identifications.

Let $\lambda$ be an Young diagram which fits into a $(r+1) \times k$ box.   Let $\lambda^c$ be the complement of the Young diagram $\lambda$. We consider the line bundles $\mathcal{L}_{\lambda^c}$ and $\mathcal{L}_{\lambda}$ on $\Fl(W)$ and $\Fl(W^*)$ as defined in Section ~\ref{classic}.
The line bundle $\mathcal{L}_{\lambda}$ pulls back to $\mathcal{L}_{\lambda^c}\tensor \det(W^*)^{k}$ under the natural map from $\Fl(W)$ to $\Fl(W^*)$, equivariantly under the action of groups
$\SLL(W)\to \SLL(W^*)$, where we note that the line bundle on $\Fl(W)$ with fiber $\det(W)$ is trivial as an equivariant  $\SLL(W)$ line bundle.

\subsection{Classical cohomology of the Grassmannian}
Recall  that the cohomology $H^*(X,\Bbb{Z})$ of a Grassmannian $X=\Gr(r+1,M)=\Gr(r+1,\Bbb{C}^M)$, $M=r+1+\ell$ is a commutative and associative ring,  with an additive   basis of  cycle classes  $\sigma_{\mu}$ of Schubert varieties $\Omega_{\mu}(E_{\bull})$. Here $\mu$ runs through all Young diagrams which fit into a $(r+1)\times \ell$ grid ($\mu$ need not be normalized).
Note that $\sigma_{\mu}\in H^{2|\mu|}(X,\Bbb{Z})$. The class of a point $[\operatorname{pt}]$ is $\sigma_{(\ell,\ell,\dots,\ell)}$.

The coefficient of $\sigma_{\lambda^c_n}$
($\lambda_n^c$ is the complement of $\lambda_n$ in a $(r+1)\times \ell$ box in the  product
$$\sigma_{\lambda_1}\cdot\sigma_{\lambda_2}\dots\cdot\sigma_{\lambda_{n-1}} \in H^*(X)=H^*(X,\Bbb{Z})$$
has the following enumerative interpretation: Pick  a general element $(E^1_{\bull},\dots,E^n_{\bull})$  of $\Fl(W)^n$
where $W=\Bbb{C}^M$, and count number of points in the intersection
$$\cap_{i=1}^{n}\Omega_{\lambda_i}(E^i_{\bull}).$$
One counts this as zero if this number is infinite.

  \subsection{Quantum Cohomology of the Grassmannian}\label{rewind}
We suggest \cite{Bertram} for the basic notions of quantum cohomology. As an abelian group, the quantum cohomology group of  $X=\Gr(r+1,M)=\Gr(r+1,\Bbb{C}^M), M=(r+1)+k$ is
 $$QH^*(X)=QH^*(X,\Bbb{Z})=H^*(X,\Bbb{Z})\tensor\Bbb{Z}[q].$$
 The multiplication in this graded commutative and associative ring is $q$-linear ($\sigma_\nu$ has degree $|\nu|$ and $q$ has degree $M$). The coefficient of $q^d\sigma_{\lambda^c_n}$ ($\lambda_n^c$ is the complement of $\lambda_n$ in a $(r+1)\times k$ box) in the quantum product
$$\sigma_{\lambda_1}\star\sigma_{\lambda_2}\star\dots\star\sigma_{\lambda_{n-1}} \in QH^*(X)$$
has the following enumerative interpretation: Fix distinct points $p_1,\dots,p_n\in\pone$, and  a general element $(E^1_{\bull},\dots,E^n_{\bull})$  of $\Fl(W)^n$
where $W=\Bbb{C}^M$, and count the number of maps 
$f:\pone\to X$ of degree $d$ so that
$$f(p_i)\in \Omega_{\lambda_i}(E^i_{\bull})\subseteq X, \forall i\in[n],$$  where if this number is infinite, the count is considered zero.
 Setting $q=0$ in the product $\sigma_{\mu}\star \sigma_{\nu}$, we recover the classical product $\sigma_{\mu}\cdot\sigma_{\nu} \in H^*(X)$.

\section{Witten's theorem}\label{WittenEnumerative}

\begin{defi}\label{evenlySplit}
 A vector bundle $\mw$ on $\pone$ is said to be evenly split
if $\mathcal{W}=\oplus_{j=1}^N\mathcal{O}_{\Bbb{P}^1}(a_j)$  with
$|a_i-a_j| \leq 1$ for $0<i<j\leq N$.
\end{defi}

\begin{remark}It is easy to see that there are evenly split bundles of every degree and (non-zero) rank. These bundles are generic bundles of the given degree and rank.
\end{remark}

\begin{theorem}[Witten's Dictionary]\label{BelkaleWitten}  Given $\vec{\lambda}\in P_{\ell}(\sL_{r+1})^n$, write $\sum_{i=1}^n|\lambda_i|=(r+1)(\ell+s) \in \Bbb{Z}$.
Let $\mw$ be an evenly split vector bundle of rank $r+1+\ell$ and degree $s$, and $\me\in \Fl_S(\mw)$ a general point.
Then $\op{rk}\Bbb{V}_{\sL_{r+1},\vec{\lambda},\ell}$ is equal to the cardinality of the set
 $$\left\{\mbox{ subbundles } \mathcal{V} \subset \mathcal{W} :  \begin{array}{lll}
& \operatorname{deg}(\mathcal{V})=0, \ \ \operatorname{rk}(\mathcal{V})=r+1, \mbox{ and } \\
& \mv_{p_i}\in\Omega_{\lambda_i}(E^{p_i}_{\bull})\subseteq \Gr(r+1,\mw_{p_i}), \ i\in[n]
\end{array}\right\},$$
\noindent
which is finite, by Kleiman transversality (see Section 2.4 of \cite{b4} for the application of Kleiman's theorem \cite{Kleiman}).
\end{theorem}
\begin{remark}
See \cite [Theorem 3.6, Eq (3.10) and Remark 3.8]{b4} for a proof of Witten's dictionary, modeled on the proof of its classical counterpart described just after Remark \ref{classicalRk} below (cf. \cite[Section 6.2]{fulton}). 
\end{remark}

\begin{remark}
In the original form of Witten's Dictionary one has quantum cohomology structure coefficients on one side (at any degree $d$), and the ranks of Verlinde bundles at degree $-d$ (the underlying bundle  of parabolic bundles have degree $-d$ which may not be zero). We choose to (cyclically) twist both sides so that the plain degree of the parabolic bundles is now zero and give conformal blocks as in our paper (~\cite{b4} explains how to twist the enumerative side so that instead of counting subbundles of the trivial bundle (as in quantum cohomology), one counts subbundles of arbitrary evenly split bundles).
\end{remark}

\subsection{Witten's Dictionary, classical analogue}\label{classicalRk}\label{class}
The classical counterpart of Witten's Dictionary, which goes back at least to L. Lesieur \cite{Le} (also \cite[Section 6.2]{fulton})  is the following. Let $\lambda_1,\dots,\lambda_n$  be Young diagrams that each fit into a $(r+1) \times \ell$ box (not necessarily normalized), and $V_{\lambda_i}$ the irreducible finite dimensional representation of $\frg$ with highest weight $\lambda_i$.
Suppose further that $\sum_{i=1}^n |\lambda_i|= (r+1) \ell$. Then the dimension of the vector space of co-invariants
$(V_{\lambda_1}\tensor \dots\tensor V_{\lambda_n})_{\sL_{r+1}}$ equals the number of points in
$$\bigcap_{i=1}^n\Omega_{\lambda_i}(E^i_{\bull})\subseteq \Gr(r+1,W),$$
where $W$ is a vector space of dimension $(r+1)+\ell$, and $(E^1_{\bull},\dots,E^n_{\bull})$ is a general element of $\Fl(W)^n$.  We often refer to this dimension as the rank of the constant bundle 
$\mathbb{A}_{\sL_{r+1},\vec{\lambda}}=(V_{\lambda_1}\tensor \dots\tensor V_{\lambda_n})_{\sL_{r+1}}\times \ovop{M}_{0,\op{n}}$ introduced earlier.

We note that the dimension of the vector space of classical co-invariants above also equals the multiplicity of the $\sL_{r+1}$ (or $\operatorname{GL}_{r+1}$) representation $V_{\lambda_n^c}$
in the tensor product $V_{\lambda_1}\tensor\dots\tensor V_{\lambda_{n-1}}$. Here $\lambda_n^c$ is the complement of the Young diagram $\lambda_n$ in a
$(r+1)\times \ell$ grid (and flipped over) and corresponds to the dual of $V_{\lambda_n}$ as a $\sL_{r+1}$ representation. Note that $r+1$ and $\ell$ are  fine-tuned to the
data of $\lambda_i$ by our assumption that $\sum_{i=1}^n |\lambda_i|= (r+1) \ell$.

\subsection{Cohomological form of Witten's dictionary}\label{quantumrule}\label{eleven}

To determine the rank of a conformal block $\vr$ in terms of quantum cohomology of Grassmannians, write
$\sum |\lambda_i| =(r+1)(\ell +s)$, and by \cite[Section 3.5]{b4}, one proceeds as follows.

 If $s<0$, then the rank of $\vr$ coincides with the rank of classical coinvariants $\clar$.

If $s\geq 0$, let $\lambda=(\ell,0,\dots,0)$. The rank of $\vr$ is the coefficient of $q^{s}[\operatorname{pt}]=q^s\sigma_{(\ell,\ell,\dots,\ell)}$  in
$$\sigma_{\lambda_1}\star\dots\star\sigma_{\lambda_n}\star\sigma_{\lambda}^{s}\in QH^*(Y),\  Y=\Gr (r+1, r+1+\ell),$$
where  $\sigma_{\lambda}^{s}$ is the $s$-fold quantum $\star$ product of $\sigma_{\lambda}$.
We can write the above multiplicity also as the coefficient of $q^s\sigma_{\lambda^c_n}$ ($\lambda^c_n$ is the complement of $\lambda_n$ in a $(r+1)\times \ell$ box) in
$$\sigma_{\lambda_1}\star\sigma_{\lambda_2}\star\dots\star\sigma_{\lambda_{n-1}}\star\sigma_{\lambda}^{s}\in QH^*(Y).$$

\begin{remark} \label{notateo} The coefficient of of $q^{s}[\operatorname{pt}]$  in the quantum product 
$\sigma_{\lambda_1}\star\dots\star\sigma_{\lambda_n}$ is denoted by
$\langle \sigma_{\lambda_1},\sigma_{\lambda_2},\dots,\sigma_{\lambda_n}\rangle_s$.
\end{remark}

  \section{The proof of vanishing above the critical level}\label{newthingproof}
We now prove Proposition ~\ref{newthing}.
 Let $\tilde{\ell}=\ell(\sL_{r+1},\vec{\lambda})+1$. We divide the proof into two cases:
 \begin{enumerate}
 \item [(1)] $\lambda_1,\dots,\lambda_n$ are in $P_{\tilde{\ell}}(\sL_{r+1})$, 
 \item[(2)] Some $\lambda_i$, say $\lambda_1$, is not in $P_{\tilde{\ell}}(\sL_{r+1})$.
 \end{enumerate}
In case (1), it suffices to prove  that $\rk \Bbb{V}_{\sL_{r+1},\vec{\lambda},\tilde{\ell}}= \rk\clar$. The enumerative translation of $\rk \Bbb{V}_{\sL_{r+1},\vec{\lambda},\tilde{\ell}}$ is the following (see Section \ref{WittenEnumerative}): Let $\mw$ be an evenly split  bundle of rank $N=r+\tilde{\ell}+1$ and degree $s=0$, i.e., $\mw=\mathcal{O}^{\oplus N}=W\tensor \mathcal{O}$ for an $N$-dimensional vector space $W$. Pick a general point $\me\in \Fl_S(\mw)$ (see Section \ref{notate}). The enumerative problem is to count (a finite list by Kleiman-transversality) subbundles $\mv$ of $\mw$ of degree $0$ and rank $r+1$ such that
$$\mv_{p_i}\in\Omega_{\lambda_i}(E^{p_i}_{\bull})\subseteq \Gr(r+1,\mw_{p_i}),\ i\in[n].$$
Any such subbundle is clearly trivial, that is, of the form $V\tensor \mathcal{O}$ for some $r+1$-dimensional subspace $V\subseteq W$, and our count therefore equals the  number of points in
$$\bigcap_{i=1}^n\Omega_{\lambda_i}(E^i_{\bull})\subseteq \Gr(r+1,W),$$
where $\rk W=N$, and $(E^1_{\bull},\dots,E^n_{\bull})$ is a general element of $\Fl(W)^n$. By Section \ref{class}, this classical count is $\rk\clar$ and we are done.
\begin{remark}\label{ClassicalRank}
We could have argued cohomologically as follows: The rank of $\Bbb{V}(\sL_{r+1},\vec{\lambda},\tilde{\ell})$ is, by Section \ref{quantumrule},
the coefficient of $[\operatorname{pt]}$ in the quantum product
$$\sigma_{\lambda_1}\star\sigma_{\lambda_2}\star\dots\star\sigma_{\lambda_{n}}\in QH^*(Y),\  Y=\Gr(r+1,r+1+\tilde{\ell}),$$
since $s=0$ in our case. This is the classical part of quantum cohomology, and our coefficient equals rank of $\clar$, as desired.
\end{remark}

In case (2), we claim that  $\clar$ is zero, and hence all $\sL_{r+1}$ conformal blocks bundles and divisors for $\vec{\lambda}$ are zero.  Assume that $k=\lambda^{(1)}_1>\tilde{\ell}$ is the maximum of the $\lambda^{(1)}_i$ and write
$\sum_{i=1}^n|\lambda_i|=(r+1)k-p(r+1)$ so that $\lambda_i$ fit into boxes of size $(r+1)\times k$ and $p>0$.
Therefore if we set $\mu_{1}=\dots=\mu_{p}=(1,1,\dots,1)$ then $\sum_{i=1}^{n}|\lambda_i|+\sum_{j=1}^p|\mu_j|=(r+1)k$.
The representation $V_{(1,\dots,1)}$ of $\operatorname{GL}(r+1)$ is trivial as a representation of $\sL_{r+1}$. Therefore applying the classical theorem relating intersection numbers for
$X=\Gr(r+1,r+1+k)$ and invariants for $\sL_{r+1}$, we find that the rank of $\clar$ equals the multiplicity of the class of a point in the cup product
$$\sigma_{\lambda_1}\cdots\sigma_{\lambda_n}\cdot\sigma_{(1,1,\dots,1)}^p\in H^*(X).$$

But it is easy to see that the cup product $\sigma_{\lambda_1}\cdot\sigma_{(1,1,\dots,1)}=0\in H^*(X)$: If $V\in \Omega_{\lambda_1}(F_{\bull})\cap  \Omega_{(1,1,\dots,1)}(G_{\bull})$ with $F_{\bull},G_{\bull}$ in general position, then $V\subseteq G_{r+k}$ and $V\supseteq F_1$ (since $\lambda_1^{(1)}=k$). But $F_1$ is not a subset of $G_{r+k}$, therefore no such $V$ exists and hence the the cup product is zero. 

(We could have argued also as follows: According to the  classical Pieri rule (see e.g., \cite{fulton}), $\sigma_{\lambda}\cdot \sigma_{b,0,\dots,0}$ is a sum $\sum \sigma_{\mu}$ over $\mu$  which can obtained from $\lambda$ by addition of $b$ boxes, no two in the same column. Since the first column of $\lambda_1^T$ is already of the full length $k$,  $\sigma_{\lambda^T_1}\cdot\sigma_{(r+1,0,\dots,0)}=0\in H^*(\Gr(k,r+1+k))$, and the desired vanishing holds by Grassmann duality.)

\subsection{Parabolic semistability}\label{Z}
The duals of spaces of conformal blocks $\vr$ over points of $\operatorname{M}_{0,n}$ can be identified with  global sections of certain line bundles on suitable moduli spaces
of semi-stable parabolic bundles on $\pone$. We note that it is not necessary (but  perhaps sufficient) for all parabolic semistable bundles to have trivial underlying bundles for $\Bbb{V}_{\sL_{r+1},\vec{\lambda},\ell}$ to be equal to the trivial bundle $\clar$ (see Remark \ref{finale} for an example). As has been pointed out to us by an anonymous referee, it is instructive to determine conditions on $\ell$ so that
any semistable parabolic bundle $\mv$ with weights $|\lambda_i|/\ell$ has the trivial splitting type. If the underlying bundle $\mv$  has $\mathcal{O}(a)$ (with $a>0$) as a direct summand, and  $\mathcal{O}(a)$ meets the flags of $\mv$ generically, and for it to contradict semistability (by an easy calculation),
$1 +\frac{1}{1}\cdot 0>0+\frac{1}{(r+1)}\cdot\sum\frac{|\lambda_i|}{\ell}$,
or that $\ell> \frac{1}{r+1}\sum|\lambda_i| =c(\sL_{r+1},\vec{\lambda})+1$.

This parabolic method misses  $\ell=c(\sL_{r+1},\vec{\lambda})+1$: Proposition \ref{newthing} says that  $\Bbb{V}_{\sL_{r+1},\vec{\lambda},\ell}$ coincides with coinvariants for $\ell=c(\sL_{r+1},\vec{\lambda})+1$
whereas this argument is inconclusive. Therefore, the critical level bound is strictly better that the bound obtained via parabolic semistability by one.   In fact, for $\ell=c(\sL_{r+1},\vec{\lambda})+1$ there
are semistable parabolic bundles with non-trivial splitting type (see e.g., Example (2) in Table \ref{Table}, and Remark ~\ref{finale}). Theorem \ref{main} and many others require the critical level as defined, and the connection of the critical level to quantum cohomology is decisive. However, we believe such a parabolic semistability argument should give bounds for all Lie algebras $\frg$, and should be further developed.

We ask the following natural question:
 \begin{question}\label{GenCritLev}  Given a triple $(\mg,\vec{\lambda},\ell)$, with $\vec{\lambda} \in P_{\ell}(\mg)^n$, what are necessary and sufficient conditions so that $\rk \mathbb{V}_{\mg,\vec{\lambda},\ell}=\rk\mathbb{A}_{\mg,\vec{\lambda}}$?
\end{question}
This answer must be subtle.  For example, Corollary~\ref{strange} says that it is sufficient to take $\ell$ greater than the critical level, but that if $\ell$ is equal to the critical level to have this rank identity
it is necessary that the critical level partner bundle has rank zero.   A look at the first four lines of Table~\ref{Table} shows that sometimes this happens, while other times it does not. Another sufficient condition is given by the theta level (Remark \ref{ThetaVanishing}).

\section{Application One: Extremality of the $\mathbb{D}=\mathbb{D}_{\frg,\vec{\lambda},\ell}$ in the Nef Cone}
In this section we give
criteria (in Propositions ~\ref{extremal1} and ~\ref{extremal2}) which come from Corollary ~\ref{strange}, to detect certain so-called $\operatorname{F}$-curves (see Definition ~\ref{FCurve}) that get contracted by the associated conformal blocks maps $\phi_{\mathbb{D}}$.  This enables us to show in Section \ref{H} that the morphisms $\phi_{\mathbb{D}}$ factor through certain Hassett contractions.

\subsection{Notation}
\begin{definition}\label{FCurve}Fix a partition of $[n]=\{1,\ldots,n\}$ into four nonempty sets $N_1$, $N_2$, $N_3$, $N_4=[n]\setminus (N_1 \cup N_2 \cup N_3)$, and consider the
 morphism
$\overline{\operatorname{M}}_{0,4} \longrightarrow \overline{\operatorname{M}}_{0,\op{n}}$,   where $(C, (a_1,a_2,a_3,a_4))\mapsto (X, (p_1,\ldots,p_n))$,
where $X$ is the nodal curve obtained as follows.  If $|N_i| \ge 2$, then one glues a copy of $\mathbb{P}^1$ to the spine $(C, (a_1,a_2,a_3,a_4))$ by attaching a point
$(\mathbb{P}^1, \{p_j : j \in N_i\} \cup \{\alpha_i\} ) \in M_{0,|N_i|+1}$
 to $a_i$ at $\alpha_i$.  If $|N_i|=1$, one does not glue any curve at the point $a_i$, but instead labels $a_i$ by $p_i$.    We refer to any element of the numerical equivalence class of the image of this morphism  the $\operatorname{F}$-Curve $F(N_1,N_2,N_3)$ or by $F(N_1,N_2,N_3,N_4)$, depending on the context.
\end{definition}

\begin{conjecture}[The $\operatorname{F}$-Conjecture]\label{F}
A divisor $D$ on $\overline{\operatorname{M}}_{0,\op{n}}$ is nef if and only if it non-negatively intersects all $\operatorname{F}$-curves.
\end{conjecture}

\subsection{Contraction results}
\begin{proposition}\label{extremal1}
 Suppose that  $r\ge 1$, $\ell \ge 1$, and let  $\vec{\lambda}=(\lambda_1,\ldots,\lambda_n)$ be an $n$-tuple in  $P_{\ell}(\sL_{r+1})$.
 Let $N_1$,$N_2$,$N_3$,$N_4$ be any partition of $[n]$ into four nonempty subsets, ordered so that, without loss of generality, if $\lambda(N_j)=\sum_{i\in N_j}|\lambda_i|$, then
$\lambda(N_1) \le \cdots \le \lambda(N_4)$.
  If
 $\sum_{j \in \{1,2,3\}}\lambda(N_j) \le r+\ell$,
 then the (possibly constant) morphism $\phi_{\mathbb{D}}$ given by the divisor
 $\mathbb{D}=\mathbb{D}_{\sL_{r+1},\vec{\lambda},\ell}$ contracts  $F(N_1,N_2,N_3)$ and in particular,  $\mathbb{D}$ is extremal in the nef cone.
  \end{proposition}

 \begin{proof}
 The  $\operatorname{F}$-curve $F(N_1,N_2,N_3)$ is contained in the boundary divisor
$$\Delta_{N_1\cup N_2 \cup N_3} \cong \overline{\operatorname{M}}_{0,|N_1\cup N_2 \cup N_3|+1} \times \overline{\operatorname{M}}_{0,|N_4|+1},$$ and is actually isomorphic to a curve contained in
$\overline{\operatorname{M}}_{0,|N_1\cup N_2 \cup N_3|+1}$ under the attaching map, which attaches the leg with $N_4$ marked points to the extra attaching point.
Therefore, to show that $\phi_{\mathbb{D}}$ contracts $F(N_1,N_2,N_3)$, it suffices to show that $\mathbb{D}$ is trivial on  $\overline{\operatorname{M}}_{0,|N_1\cup N_2 \cup N_3|+1}\times x$ for a fixed $x\in \overline{\operatorname{M}}_{0,|N_4|+1}$.  Let
$I=N_1\cup N_2 \cup N_3$.

When pulled back to $M=\overline{\operatorname{M}}_{0,|N_1\cup N_2 \cup N_3|+1} \times \overline{\operatorname{M}}_{0,|N_4|+1}$, our bundle $\mathbb{V}_{\sL_{r+1},\vec{\lambda}, \ell}$  breaks up (by factorization) into a direct sum
$$\oplus_{\mu\in {P}_{\ell}}\mathbb{V}_{\sL_{r+1},\{\lambda_i : i \in I\} \cup \{\mu\},\ell}\tensor \mathbb{V}_{\sL_{r+1},\{\lambda_i : i \in N_4\} \cup \{\mu^*\}, \ell}$$
of tensor products of vector bundles pulled back from the two projections of $M$. It therefore suffices to show that $c_1(\mathbb{V}_{\sL_{r+1},\{\lambda_i : i \in I\} \cup \{\mu\},\ell})=0$ for any $\mu\in P_{\ell}$. But  $|\mu|\leq\ell r$ and
$\sum_{i\in I}| \lambda_i| + |\mu|\leq \ell +r +\ell r<(\ell+1)(r+1)$
and hence $\ell$ is greater than the critical level for $(\sL_{r+1},\{\lambda_i : i \in I\} \cup \{\mu\})$. We may therefore apply Proposition \ref{newthing} to conclude that the desired $c_1$ vanishes.
 \end{proof}

Using Remark \ref{ThetaVanishing} in place of Corollary ~\ref{strange}, gives the following general result for arbitrary $\frg$.
 \begin{proposition}\label{extremal2}
 Suppose that  $r\ge 1$, $\ell \ge 1$, and let  $\vec{\lambda}=(\lambda_1,\ldots,\lambda_n)$ be an $n$-tuple in $P_{\ell}(\frg)$.  Let $N_1$,$N_2$,$N_3$,$N_4$ be any partition of $[n]$ into four nonempty subsets, ordered so that, without loss of generality,
 if $L(N_j) = \sum_{i\in N_j} (\lambda_i,\theta)$, then $L(N_1) \le L(N_2) \le L(N_3) \le L(N_4)$.
 If
 $\sum_{j\in \{1,2,3\}}L(N_j) \leq \ell + 1$,
 then the (possibly constant) morphism $\phi_{\mathbb{D}}$ afforded by the divisor
 $\mathbb{D}=\mathbb{D}_{\frg,\vec{\lambda},\ell}$ contracts the $\operatorname{F}$-Curve $\operatorname{F}(N_1,N_2,N_3)$, and in particular,  $\mathbb{D}$ is extremal in the nef cone.
 \end{proposition}

\section{Application Two: Conformal blocks morphisms and Hassett contractions}\label{H}
As a step towards understanding the images of the morphisms $\phi_{\mathbb{D}}$, for $\mathbb{D}=c_1(\mathbb{V}(\mg,\vec{\lambda}, \ell))$, we show that if the weights $\lambda_i$ satisfy prescribed conditions (as given in Theorems \ref{Smallerd} and \ref{Smallerdnew}), then $\phi_{\mathbb{D}}$ factors through
maps to Hassett's moduli spaces $\overline{\operatorname{M}}_{0,\mathcal{A}}$ of stable weighted pointed rational curves, where the weight data is explicitly determined by the $\frg$, $\ell$ and $\lambda_i$.

\subsection{Background on Hassett spaces}\label{Hassett} Consider
an $n$-tuple  $\mathcal{A}=\{a_1,\ldots,a_n\}$, with $a_i \in \mathbb{Q}$, $0 < a_i \le 1$, such that $\sum_i a_i >2$.   In \cite{HassettWeighted}, Hassett introduced moduli spaces $\overline{\operatorname{M}}_{0,\mathcal{A}}$, parameterizing families of stable weighted pointed rational curves $(C, (p_1,\ldots,p_n))$
such that (1) $C$ is nodal away from its marked points $p_i$; (2) $\sum_{j\in J} a_i \le 1$, if the marked points $\{p_j : j\in J\}$ coincide; and (3) If $C'$ is an irreducible component of $C$ then
$\sum_{p_i \in C'} a_i + \mbox{ number of nodes on } C' > 2$.  These Hassett spaces  $\overline{\operatorname{M}}_{0,\mathcal{A}}$ receive birational morphisms $\rho_{\mathcal{A}}$ from
$\overline{\operatorname{M}}_{0,\op{n}}$ that are characterized entirely by which F-Curves (see Def. \ref{FCurve}) they contract.

\begin{definition/lemma}\label{rho}For any Hassett space $\overline{\operatorname{M}}_{0,\mathcal{A}}$, with  $\mathcal{A}=\{a_1,\ldots,a_n\}$, there are birational morphisms
$\rho_{\mathcal{A}}: \overline{\operatorname{M}}_{0,\op{n}} \longrightarrow  \overline{\operatorname{M}}_{0,\mathcal{A}}$,
contracting all  $\operatorname{F}$-curves $F(N_1,N_2,N_3,N_4)$ satisfying: $\sum_{i\in N_1 \cup N_2 \cup N_3} a_i \le 1$,
and no others, where without loss of generality, the leg $N_4$ carries the most weight.
\end{definition/lemma}

\subsection{Results on Hassett spaces}
The following theorems generalize  \cite[Proposition 4.7]{Fakh}, where $\frg=\sL_2$ was considered.
 \begin{theorem}\label{Smallerd} Let $\mathbb{D}=\mathbb{D}_{\sL_{r+1},\vec{\lambda},\ell}$ be such that:
\begin{enumerate}
\item $0<|\lambda_i|\le \ell+r$ for all $i \in \{1,\ldots,n\}$;
\item  $\sum_{i=1}^n|\lambda_i| > 2(r+\ell)$.
\end{enumerate}
Then the morphism $\phi_{\mathbb{D}}$   factors through $\rho_{\mathcal{A}}: \overline{\operatorname{M}}_{0,\op{n}} \longrightarrow \overline{\operatorname{M}}_{0,\mathcal{A}}$,  where
  $\mathcal{A}=\{a_1,\ldots,a_n\}$, $a_i=\frac{|\lambda_i|}{r+\ell}$.
\end{theorem}

\begin{proof}
For $\mathcal{A}=\{a_1,\ldots,a_n\}, \ \ a_i=\frac{|\lambda_i|}{r+\ell},$
as in the hypothesis, the condition $|\lambda_i|\le \ell+r$, guarantees that $a_i \le 1$ for all $i$, and $\sum_{i=1}^n|\lambda_i| > 2(r+\ell)$ guarantees that $\sum_{i=1}^na_i>2$.

By \cite[Lemma 4.6]{Fakh}, we need to show that any $\operatorname{F}$-curve $F(N_1,N_2,N_3,N_4)$ contracted by   $\rho_{\mathcal{A}}$ is also contracted by
 $\phi_{\mathbb{D}}$. Suppose that  $\rho_{\mathcal{A}}$ contracts the $\operatorname{F}$-curve $F(N_1,N_2,N_3,N_4)$, so that in particular, by Definition/Lemma \ref{rho}, $\sum_{i\in N_1 \cup N_2 \cup N_3}a_i \le 1$. Then $\sum_{j \in \{1,2,3\}}\lambda(N_j) =(r+\ell)\sum_{j\in \{1,2,3\}} a_j \le r+\ell$,
and hence by Proposition \ref{extremal1},  $\phi_{\mathbb{D}}$ contracts   $\operatorname{F}$-curve $F(N_1,N_2,N_3,N_4)$.
\end{proof}
Using Proposition \ref{extremal2} in place of Proposition \ref{extremal1}, we get
 \begin{theorem}\label{Smallerdnew} Let $\mathbb{D}=\mathbb{D}_{\frg,\vec{\lambda},\ell}$ be such that
 $\lambda_i\neq 0$ for all $i$ and
 \begin{equation}\label{red}
 \sum_{i=1}^n \lambda_i(H_{\theta})> 2 (\ell+1).
 \end{equation}

Then the morphism $\phi_{\mathbb{D}}$   factors through $\rho_{\mathcal{A}}: \overline{\operatorname{M}}_{0,\op{n}} \longrightarrow \overline{\operatorname{M}}_{0,\mathcal{A}}$, where
  $\mathcal{A}=\{a_1,\ldots,a_n\}$, $a_i=\frac{(\lambda_i,\theta)}{\ell+1}.$
\end{theorem}
Note that if $\mathbb{D}=\mathbb{D}_{\frg,\vec{\lambda},\ell}$ is non trivial then by Remark \ref{ThetaVanishing}, $\sum (\lambda_i,\theta)\geq 2 (\ell+1)$.

\section{Table}\label{Table}
In these examples, computed using \cite{ConfBlocks},  $\ell$ is a critical level for the pair $(\sL_{r+1},\vec{\lambda})$, and Deg (when $n=4$), denotes the degree of $\vr$ (and of $\vl$)
over $\overline{\operatorname{M}}_{0,4}=\pone$.  In particular, Deg $=0$ means $\mathbb{D}_{\sL_{r+1},\vec{\lambda},\ell}=\mathbb{D}_{\sL_{\ell+1},\vec{\lambda}^T,r}=0$, even when the rank of $\vr$ (and of $\vl$) is nonzero.   A $*$ in the Deg column indicates more than $4$ marked points.
\begin{center}
    \begin{tabular}{ |l|l| l | p{7.3 cm} | l | l|l|}
    \hline
    \text{Deg}& $(r+1, \ell+1)$ & $n$ & $ \  \ \ \ \ \ \ \vec{\lambda}=(\lambda_1,\dots, \lambda_n)$  \ \ & $\rk \mathbb{A}$ & $\rk\mathbb{V}_{\sL_{r+1},\vec{\lambda},\ell}$ & $\rk\mathbb{V}_{\sL_{\ell+1},\vec{\lambda}^T,r}$ \\ \hline

    *&$(3,2)$      & $6$ & $(\omega_1,\dots ,\omega_1)$  & 5 & 1 & 4\\ \hline
    1&$(3,2)$      & $4$ & $(\omega_1,\omega_1, \omega_2, \omega_2)$ & 2 & 1 & 1 \\ \hline
    0&$(4,4)$      & $4$ & $(\omega_1, (2\omega_1+\omega_3)^3)$ & 2&1 & 1\\ \hline
    *& $(3,6)$      & $5$ & $(2\omega_1+ \omega_2, \omega_2, 2\omega_1, 2 \omega_2, 3\omega_2)$ & 7 &7 & 0\\ \hline
    *& $(3,5)$      & $5$ & $(2\omega_1+ \omega_2, \omega_2, 2\omega_1, 2 \omega_2, \omega_1 + \omega_2)$ & 9 &8 & 1\\ \hline
  0&$(4,4)$      & $4$ & $(\omega_2+\omega_3, \omega_1, \omega_1+2\omega_2, 2\omega_1+\omega_3)$ & 2&1 & 1\\ \hline
   0&$(4,5)$      & $4$ & $(\omega_1,2\omega_1+\omega_2+\omega_3,(3\omega_1+\omega_3)^2)$ & 2&1 & 1\\ \hline
    1& $(4,5)$ &4& $(\omega_1+\omega_3, 2\omega_1+2\omega_2, 2\omega_1+2\omega_2, 4\omega_1)$ &4 &1 & 3 \\ \hline
     *& $(3,6)$ &8& $((2\omega_1)^6, \omega_2, 2\omega_2)$ & 150 &136 & 14\\ \hline

    \end{tabular}
\end{center}

\section{The proof of Theorem \ref{main}: The main reductions}\label{sketch}
For the proof of Theorem \ref{main}, in addition to  \cite{witten,agni}, we use the  geometric procedure of creating invariants from Schubert calculus \cite{PB}. Together with a standard formalism, this leads to
a duality isomorphism of classical invariant theory (reviewed in Section \ref{classic}):
 \begin{equation}\label{duality}
 \aclar^*\leto{\sim} \aclal.
 \end{equation}
The proof of Theorem \ref{main} then breaks up into two parts:
\begin{enumerate}
\item We prove the equality \eqref{morpheusnew} by working over $\operatorname{M}_{0,n}$ and using the connection between quantum cohomology and conformal blocks, and a degeneration argument for Gromov-Witten invariants (but working over a fixed point of $\operatorname{M}_{0,n}$).
\item We show that the natural morphism  $\vr^*\to \vl$  arising from  the isomorphism ~\eqref{duality}  is the zero map. To do so, we invoke the interpretation \cite{pauly} of $\vr^*$  as global sections of  a line bundle on a moduli space of parabolic bundles (valid over $\operatorname{M}_{0,n}$, but not over  $\overline{\operatorname{M}}_{0,\op{n}}$). Again, it suffices to work over $\operatorname{M}_{0,n}$.
\end{enumerate}
These two steps lead to an isomorphism  ~\eqref{morpheus}. Let $\Bbb{K}_{r+1,\ell,\vec{\lambda}}$ be the kernel of $\clar\to \vr$: Consider the exact
sequences
$$0\to \vr^* \to \clar^*\to \Bbb{K}_{r+1,\ell,\vec{\lambda}}^*\to 0$$
and
$$0\to \Bbb{K}_{\ell+1,r,\vec{\lambda}^T}\to \clal\to \vl\to 0.$$
Step (2) gives rise  to vertical arrows from the first exact sequence above to the second. We therefore find a map  ~\eqref{morpheus}. By the snake lemma, ~\eqref{morpheus} is a surjective morphism of vector bundles of the same rank, hence an isomorphism.

\subsection{Classical strange duality and the map ~\eqref{duality}}\label{classic}

We recall from \cite{PB}, how Schubert calculus of Grassmannians can be used to produce natural bases in the spaces of
invariants and hence produces the duality isomorphism ~\eqref{duality} by the (standard) strange duality formalism. The following is in analogy with the tangent space of a Schubert cell $\Omega^o_{\lambda}(E_{\bull})$ in a Grassmannian (see Section ~\ref{added2}).

Let $V$ and $Q$ be vector spaces of dimensions $r+1$ and $\ell +1$ respectively. Let $N=r+\ell +2$ (as before).
Let $\vec{\lambda}=(\lambda_1,\dots,\lambda_n)$ be an $n$-tuple of  Young diagrams that each fit into a $(r+1)\times (\ell+1)$ box. Assume $\sum_{i=1}^n |\lambda_i| = (r+1)(\ell+1)$. We define a divisor in $D\subseteq \Fl(V)^n\times \Fl(Q)^n$: A point $(F^1_{\bull},\dots, F^n_{\bull}, G^1_{\bull},\dots, G^n_{\bull})\in D$ iff $\exists$ a non-zero map
$\phi:V\to Q$ so that for $a\in[r+1]$ and $i\in[n]$,
\begin{equation}\label{defD}
\phi(F^i_{a})\subseteq G^i_{\ell+1-\lambda_i^{(a)}}.
\end{equation}

We will make this precise by defining $D$ as a determinantal scheme (see Section \ref{books4}) and also identify $\mathcal{O}(D)$. We choose to write the answers in a symmetric manner. We set
$T=Q^*$. Note that $\Fl(Q)$ is canonically identified with $\Fl(T)$. For $a\in[r+1]$, let $\ml_a\in \operatorname{Pic}(\Fl(V))$ be the pull back of the ample generator (top exterior power of the dual of the universal subbundle) of $\Gr(a,V)$ by the tautological map
$$\Fl(V)\to \Gr(a,V),\ F_{\bull}\mapsto F_a.$$

For non-negative integers $\nu_1,\dots,\nu_{r+1}$, define a Young diagram $\lambda=(\lambda^{(1)},\dots,\lambda^{(r+1)})$ by
$$\lambda^{(a)}=\nu_a+\nu_{a+1}+\dots+\nu_{r+1}$$
and a line bundle $\ml_{\lambda}=\ml_1^{\nu_1}\tensor\dots\tensor\ml_{r+1}^{\nu_{r+1}}$, whose fiber over $F_{\bull}$ is denoted by $\ml_{\lambda}(V,F_{\bull})=\ml_{\lambda}(F_{\bull}).$

\begin{proposition}(Borel-Weil)
The following are isomorphic as representations of $\GL(V)$:
$$H^0(\Fl(V),\ml_{\lambda})= V_{\lambda}^*.$$
\end{proposition}

\subsubsection{}\label{refcomment}
For a Young diagram $\lambda$,  define a  vector bundle $\mathcal{P}_{\lambda} $ on $\Fl(V)\times\Fl(Q)$ with fiber over $(F_{\bull},G_{\bull})$ given by
 $$\{ \phi \in \Hom (V,Q) | \phi(F_a) \subseteq G_{\ell+1-\lambda^{(a)}}, a\in[r+1]\},$$
 and $\mathcal{T}_{\lambda}= \Hom(V,Q)/\mathcal{P}_{\lambda}.$
  
 \begin{itemize}
 \item The rank of $\mathcal{T}_{\lambda} $ equals $|\lambda|$.
 \item The  fiber of the determinant of $\mathcal{T}_{\lambda} $ at a point $(F_{\bull},G_{\bull})$ equals
 $\ml_{\lambda}(F_{\bull}) \tensor (\det Q)^{r+1}\tensor \ml_{\mu}(G_{\bull}),$
 where $\mu=(\mu^{(1)},\dots,\mu^{(\ell+1)}) $ is the partition  transpose to the partition $(\ell+1-\lambda^{(r+1)},\dots,\ell+1-\lambda^{(1)})$. (See ~\cite{PB}, and \cite{bk2}, formula 16, where $\mu$ is called the $(\ell+1)$-flip of $\lambda$.)
 \item If $\widetilde{G}_{\bull}\in \Fl(T)$ corresponds to $G_{\bull}\in\Fl(Q)$ under the identification $\Fl(T)=\Fl(Q)$, we can
 write the fiber of the determinant of $\mathcal{T}_{\lambda} $, symmetrically as 
 $\ml_{\lambda}(F_{\bull}) \tensor  \ml_{\lambda^T}(\widetilde{G}_{\bull}).$
 \end{itemize}
\subsubsection{}
One can view $D$ as the zero locus of the determinant of the following morphism on $\Fl(V)^n\times \Fl(Q)^n$ of vector bundles
of the same rank

$$\Hom(V,Q)\to \bigoplus_{i=1}^n \pr^*_i\mathcal{T}_{\lambda_i}$$
where $\pr_i$ are the projection maps.

Recall that if $\mathcal{V} \to \mathcal{W}$ is a map between vector bundles of the same rank on a scheme $S$, then one gets a canonical element $s\in H^0(S,\det \mw\tensor(\det \mv)^*)$.
In this way  we get a line bundle on $\Fl(V)^n\times \Fl(Q)^n=\Fl(V)^n\times \Fl(T)^n$, $T=Q^*$   with a canonical section $s$ whose zero locus is $D$. 

This line bundle has fiber at 
$(F^1_{\bull},\dots, F^n_{\bull}, \widetilde{G}^1_{\bull},\dots, \widetilde{G}^n_{\bull})\in \Fl(V)^n\times \Fl(T)^n$ (see Section \ref{refcomment}), equal to
$$\otimes_{i=1}^n(\ml_{\lambda_{i}}(F^i_{\bull})\tensor \ml_{\lambda^T_{i}}(\widetilde{G}^i_{\bull}))\otimes (\det V)^{\ell+1}\otimes(\det T)^{r+1}.$$

We find therefore an element
$$s\in H^0(\Fl(V)^n,\otimes_{i=1}^n\ml_{\lambda_{i}})^{\SLL(V)}\tensor H^0(\Fl(T)^n,\otimes_{i=1}^n \ml_{\lambda^T_{i}})^{\SLL(T)}=\aclar^*\tensor \aclal^*$$ and  a duality map
\begin{equation}\label{duality2}
\aclal\to \aclar^*.
\end{equation}

\subsection{Relations to Schubert calculus}\label{Schubert}We now explain why the above duality map \eqref{duality2} is an isomorphism. 
Let $W$ be an $N=r+\ell+2$-dimensional  vector space equipped with $n$ flags $H^1_{\bull},\dots,H^n_{\bull}$ in general position. We fix $V$ and $Q$ as above of dimensions $r+1$ and $\ell+1$ respectively.
By Kleiman transversality, the intersection $\cap_{i=1}^n\Omega_{\lambda_i}(H^i_{\bull}) \subseteq \operatorname{Gr}(r+1,N)$ is finite, and transverse. Let $V_1,\dots,V_m$  be the points in this intersection, and let $Q_1,\dots,Q_m$ the corresponding quotients,
$Q_a=W/V_a$. Each of the  vector spaces $V_1,\dots,V_m,Q_1,\dots,Q_m$  receive $n$ canonical induced flags (from $W$).
and hence give $\operatorname{GL}(V)$ and $\operatorname{GL}(Q)$-orbits in $\Fl(V)^n$ and $\Fl(Q)^n$ respectively (by choosing isomorphisms $V_a\to V$, and $Q_a\to Q$). We can dualize the $\operatorname{GL}(Q)$ orbits to obtain $\operatorname{GL}(T)$ orbits in $\Fl(T)^n$ corresponding to each of $Q_1,\dots,Q_m$ (see Section \ref{grdual}).

Choose orbit representatives $x_1,\dots,x_m\in \Fl(V)^n$ and $y_1,\dots,y_m\in \Fl(T)^n$. From the setup, it is known
that $m=\rk (\tensor_{i=1}^n V_{\lambda_i}^*)^{\operatorname{SL}(V)}=\rk(\tensor_{i=1}^n V_{\lambda_i^T}^*)^{\SLL(T)}$, and that
$s(x_a,y_b)\neq 0$ if and only if $a=b$. Therefore the sections $s(\cdot,y_a)$ (well defined up to scalars) form a basis for  $(\tensor_{i=1}^n V_{\lambda_i}^*)^{\SLL(V)}$ (this was the main result of \cite{PB}). Hence by the standard strange duality formalism, the duality map ~\eqref{duality2} is an isomorphism, also see  Lemma ~\ref{books} with quotient stacks $M=\Fl(V)^n/\operatorname{SL}(V)$ and $N=\Fl(T)^n/\operatorname{SL}(T)$ (and $\ma$ and $\mb$ are descents of
$\tensor_{i=1}^n \ml_{\lambda_{i}}$ and $\tensor_{i=1}^n \ml_{\lambda^T_{i}}$).

\begin{remark}\label{transversality}
We recall  the reason for the vanishing $s(x_a,y_b)=0$ for $a\neq b$: The natural non-zero map $\phi:V_a\to Q_b$ (inclusion in $W$ followed by projection to $Q_b$) satisfies the conditions of ~\eqref{defD}. It is also easy to see that $s(x_a,y_a)\neq 0$: If $(x_a,y_a)\in D$
then any map $\phi$ in the definition of $D$ gives us an element in $\cap_{i=1}^nT\Omega_{\lambda_i}(H^i_{\bull})_{V_a}=0$ (by transversality).
\end{remark}
The map \eqref{morpheus} is defined to be the inverse of \eqref{duality2}. By Lemma \ref{books}, we have explicit control of \eqref{morpheus} when we lay out a suitable enumerative problem.

\subsection{A universal situation}\label{books1}
We analyze the ``strange duality'' setting of schemes (or stacks) $M,N$ equipped with line bundles $\mathcal{A}$ and $\mb$ and a section $s$ of $\mathcal{A}\boxtimes\mb$ in some detail (an instance of this has appeared in Section \ref{Schubert}). We will use this setting again in Section \ref{books4}.  Fix the following data:
\begin{enumerate}
\item Let $M$ and $N$ be schemes  (or quotient stacks $X/G$) with line bundles $\mathcal{A}$ and $\mb$ respectively. Assume  $H^0(M,\ma)$ and $H^0(N,\mb)$ are
both finite dimensional of the same dimension $m$.
\item Suppose we are given a section $s$ of $\mathcal{A}\boxtimes \mathcal{B}$ on $M\times N$. This gives rise to a (possibly degenerate) ``duality''  $s\in H^0(M,\ma)\tensor H^0(N,\mb)$, or a map $D:H^0(M,\ma)^*\to  H^0(N,\mb)$.
\end{enumerate}

Now suppose we are able to manufacture points $x_1,\dots, x_m\in M$ and points $y_1,\dots,y_m\in N$ so that $s(x_a,y_b)=0$ if $a\neq b$ and non-zero (as an element of $\ma_{x_a}\tensor\mb_{y_a}$)
if $a=b$. Set $\alpha_a= s(x,y_a)\in H^0(M,\ma)\tensor \mb_{y_a}$ and $\beta_a=s(x_a,y)\in \ma_{x_a}\tensor H^0(N,\mb)$.

\begin{lemma}\label{books}
\begin{enumerate}
\item
\begin{equation}\label{bb}
s=\sum _{a=1}^ m s(x_a,y_a)^{-1}\alpha_a\beta_a.
\end{equation}
\item The duality map $D:H^0(M,\ma)^*\to  H^0(N,\mb)$ is an isomorphism. It carries the element (well defined up-to scalars) ``evaluation at $x_a$"  in $H^0(M,\ma)^*$ to the section $\beta_a$ (up-to scalars).
\end{enumerate}
\end{lemma}
\begin{proof}
From the given data, it is clear that $\{\alpha_a:a\in[m]\}$ (resp.  $\{\beta_a:a\in[m]\}$) are linearly independent (we can ignore the twists by constant lines) and hence form a basis of
$H^0(M,\ma)$ (resp. $H^0(N,\mb)$) respectively. Therefore there is an expression of the form
$$s=\sum_{a=1}^m\sum_{b=1}^m \gamma_{a,b} \alpha_a\beta_b$$
with $\gamma_{a,b}\in \ma_{x_a}^{-1}\tensor\mb_{y_b}^{-1}$. Evaluating this equation at  points of the form $(x_a,y_b)$ , we obtain ~\eqref{bb}. It is easy to see that (ii) follows from
\eqref{bb}.
\end{proof}

\section{Proof of Proposition \ref{strange}, Part {\textnormal{a}}}\label{setting}

An outline of the proof of Proposition \ref{strange} (a), i.e., the equality \eqref{morpheusnew}, is provided in Section \ref{outline}, after some notation has been introduced. The basic idea is to obtain enumerative interpretations
for the ranks of $\clar$, $\vr$ and $\vl$. The reckoning for $\clar$ admits deformations: in a particular degenerate situation, the enumerative problem breaks up into two parts which are identified with the ranks of $\vr$ and $\vl$, and Proposition \ref{strange}, (a) follows.

We first define a number of enumerative problems.   We set $N=r+\ell+2$, and introduce the notation
$I^{p}=\{i^p_1<\dots<i^p_{r+1}\}$ for $p\in S=\{p_1,\dots,p_n\}$,
by the formula  $\lambda_{i}^{(a)}=(\ell+1)+a-i^{p_i}_a$ for $a\in[r+1], i\in[n]$. Since $\vec{\lambda}$ is an $n$-tuple of normalized weights, $i^p_{r+1}=N$ and $i^p_1\neq 1$. We are next going to establish three enumerative interpretations: for $\rk\clar$, $\rk\vr$ and $\rk\vl$.

\subsubsection{The enumerative problem corresponding to $\rk\clar$}\label{bb3}
 We first note that $\clar$ and $\clal$ have the same rank (by their enumerative interpretation and  Grassmann duality, see Section ~\ref{grdual}) and that these ranks equal the following enumerative number: Let $\mw=\mathcal{O}^{\oplus N}$.  Choose a general point $\me\in\Fl_S(\mw)$. The enumerative problem is: Count subbundles $\mv$ of degree $0$ and rank $r+1$ of $\mw$ such that $\mv_{p_i}\in \Omega_{\lambda_i}(E^{p_i}_{\bull})\subseteq \Gr(r+1,{\mw}_{p_i})$,  $i=1,\dots,n$.

Note that this is the same problem as counting classical intersection of generic translates of Schubert varieties (corresponding to $\lambda_i$) in a  Grassmannian $\Gr(r+1,N)$.
\subsubsection{The enumerative problem corresponding to $\rk\vr$}\label{ceremonie1}
 Let $\mw'=\mathcal{O}(1)\oplus\mathcal{O}^{\oplus N-2}$. Choose a general point $\me'\in \Fl_S(\mw')$. Count subbundles $\mv$ of degree $0$ and rank $r+1$ of $\mw'$ such that $\Omega_{\lambda_i}(E'^{p_i}_{\bull})\subseteq \Gr(r+1,{\mw'}_{p_i})$, $i=1,\dots, n$.
\subsubsection{The enumerative problem corresponding to $\rk\vl$}\label{ceremonie2}
The enumerative problem corresponding to $\vl$ is of counting subbundles of $\mw'$ of degree $0$ and rank $\ell+1$ of $\mw'=\mathcal{O}(1)\oplus\mathcal{O}^{\oplus N-2}$ subject to incidence conditions at the points $p\in S$. We ``Grassmann dualize'' this problem by considering the dual of quotients $\mv=(\mw'/\mt)^*\subseteq (\mw')^*$. The association $\mt$ with $\mv$ is one-one and we may write down equivalent conditions on $\mv$.

The resulting enumerative problem is the following: Let $\overline{\mw}=\mathcal{O}(-1)\oplus\mathcal{O}^{\oplus N-2}$. Choose a general point $\overline{\me}\in\Fl_S(\overline{\mw})$. Count subbundles $\mv$ of degree $-1$ and rank $r$ of $\overline{\mw}$ such that $\mv_{p_i}\in \Omega_{\lambda_i}(\overline{E}^{p_i}_{\bull})\subseteq \Gr(r+1,\overline{\mw}_{p_i})$, $i=1,\dots, n$.

\subsection{Outline of proof of Proposition \ref{strange}, Part {\textnormal{a}}}\label{outline}
The main idea is to degenerate the enumerative problem corresponding to $\clar$ by replacing $\mathcal{O}^{\oplus N}$ by its simplest degeneration
 $\mathcal{O}(1)\oplus\mathcal{O}(-1)\oplus\mathcal{O}^{\oplus (N-2)}$. The  quot-scheme corresponding to $\mathcal{O}^{\oplus N}$ (actually a Grassmannian) degenerates (flatly) into a union of two smooth quot schemes intersecting transversally.

 For every bundle  $\mw$ of rank $N$ and degree $0$ we can pose an enumerative problem (to ensure conservation of numbers, we will work with versions of the enumerative problems where one is looking for points in projective varieties, see Section \ref{bertram}): Fix  a general $\me\in \Fl_S(\mw)$ and ``count'' coherent subsheaves $\mv\subseteq\mw$ of degree $0$ and rank $r+1$ such that $\mv_p\to \mw_p/E^p_{i^{p}_{a}}$ has kernel of dimension at least $a$ for all $a\in[r+1]$, and $p\in S$. Note that for some $\mw$ the above problem is not enumerative,
 i.e., the solution scheme is not of the expected dimension (or non-reduced). But we will show that for  $\mathcal{O}^{\oplus N}$ and $\mathcal{O}(1)\oplus\mathcal{O}(-1)\oplus\mathcal{O}^{\oplus (N-2)}$, this  problem is enumerative.

  When $\mw=\mathcal{O}^{\oplus N}$, this is the counting problem of Section \ref{bb3}. Now replace $\mathcal{O}^{\oplus N}$ by $\mathcal{O}(1)\oplus\mathcal{O}(-1)\oplus\mathcal{O}^{\oplus (N-2)}$. The degenerate enumerative number breaks up into two parts and we identify these parts as ranks of conformal blocks. We show that the enumerative counts  are conserved in Section \ref{conservation}, and hence obtain the equality \eqref{morpheusnew}.

\subsection{Quot schemes and degenerations}
Let $\quot$ be the quot scheme of degree $0$ and rank $ r+1$-subsheaves of
$\mw=\mathcal{O}(1)\oplus\mathcal{O}(-1)\oplus\mathcal{O}^{\oplus N-2}$.
 Note that the surjection $\mw\to \mathcal{O}(-1)$ and
the inclusion $\mathcal{O}(1)\to \mw$ are canonical (up-to scalars). Hence the corresponding kernel and quotients are canonical.

Consider the evenly split sheaves  $\mw'=\mathcal{O}(1)\oplus\mathcal{O}^{\oplus N-2}\subseteq \mw$ and $\tmw=\mw/\mathcal{O}(1)$. Let $\quot_1$ be (smooth) quot scheme of degree $0$ rank $r+1$ subsheaves of $\mw'$ and $\quot_2$ the (smooth) quot scheme of degree $-1$  and rank $r$ subsheaves of $\tmw$.
\begin{lemma}\label{ceremonienew}
\begin{enumerate}
\item $\quot=\quot_1\cup \quot_2$ where $\quot_1$ and $\quot_2$ are smooth subschemes and  $C=\quot_1\cap\quot_2$ is a subscheme of smaller
 dimension.
 \item $\quot$ is smooth in the complement of $\quot_1\cap\quot_2$.
 \item Points of $\quot\setminus\quot_1\cap\quot_2$ correspond to subbundles of $\mw$.
 \item The dimensions of $\quot_1$ and $\quot_2$ equal $(r+1)(\ell+1)$.
 \end{enumerate}
\end{lemma}
\subsection{Proof of Lemma ~\ref{ceremonienew}}\label{proofofceremonie}

Consider a $\mv\subseteq\mw$
such that $\mq=\mw/\mv$ is a point of $\quot$. Clearly $\mv$ cannot have factors of
the type $\mathcal{O}(a)$, $a\geq 2$ as these do not admit non-zero maps to $\mw$. Any map $\mathcal{O}(1)\to \mw$ has
image inside the standard copy  $\mathcal{O}(1)\subseteq \mw$. Therefore there cannot be more than one $\mathcal{O}(1)$ in  such a $\mv$.

The maps $\quot_j\to\quot,\ j=1,2$ are clear. Let $\mv\subseteq\mw$ correspond to a point on $\quot$. If the composite $\mv\to\mw\to\mathcal{O}(-1)$ is the zero map then the point is in $\quot_1$.
Otherwise, $\mv$ has a factor of $\mathcal{O}(1)$ which maps isomorphically to the canonical copy of $\mathcal{O}(1)$ in $\mw$. So $\mv$ is determined by the sub-sheaf $\mv/\mathcal{O}(1)\subseteq \tmw$ i.e., a point of $\quot_2$.

Points in $C$ correspond to subsheaves of $\mw'/\mathcal{O}(1)$ of degree $-1$ and rank $r$. It is easy to show the following: $$\dim\quot_1=\dim\quot_2=\dim C +1$$

Let $A=\quot_1\setminus C\subseteq \quot$ and $B=\quot_2\setminus C$. At points of $A$, $\mv$ is of the form $\mathcal{O}^{\oplus r+1}$. In this case clearly $H^1(\mathbb{P}^1, \mv^*\tensor\mw)=0$. Long exact sequences in cohomology  imply that $H^1(\mathbb{P}^1, \mv^*\tensor\mw)=0$ surjects on to  $H^1(\mathbb{P}^1, \mv^*\tensor\mw/\mv)$, and
therefore the last group is zero. Therefore $\quot$ is smooth at such points.

At points of $B$,  $\mv$ is isomorphic to $\mathcal{O}(1)\oplus\mathcal{O}(-1)\oplus\mathcal{O}^{r-1}$  and $\ms=\mathcal{O}(1)\oplus\mathcal{O}^{\oplus{N-2}}$ surjects onto $\mw/\mv$. It is easy to see that $H^1(\mathbb{P}^1, \mv^*\tensor  \ms)=0$ and hence the required vanishing $H^1(\mathbb{P}^1, \mv^*\tensor\mw/\mv)=0$ follows. Therefore $\quot$ is smooth at such points.

\subsection{Degeneration of the Gromov-Witten numbers}\label{ceremonie}
With $\mw$ as above, consider a generic point $\me\in \Fl(\mw)$. For $p\in S$, the fiber $\mw_p$ has a canonical quotient $\tau_p:\mw_p\to L_p$ (corresponding to  $\tau:\mw\to \mathcal{O}(-1)$) and a canonical line $M_p\subseteq \mw_p$ (corresponding to $\mathcal{O}(1)\subseteq \mw$). In addition, $\tau_p(M_p)=0$. Genericity of the flags implies that $E^p_{1}$ surjects onto $L_p$ via $\tau_p$ and $E^p_{1}\cap M_p=\{0\}$. The induced flags $\me'\in \Fl_S(\mw')$ (note $\ker\tau=\mw'$) are therefore suitably general. Note that $E'^{p}_{\bull}$ is the flag
$$0\subseteq E^p_2\cap \mw'_p\subseteq E^p_3\cap\mw'_p\subseteq\dots\subseteq E^p_N\cap\mw'_p=\mw'_p.$$

Similarly the induced flags $\overline{\me}\in \Fl_S(\overline{\mw})$ (note $\overline{\mw}=\mw/\mathcal{O}(1)$) are suitably generic. Note that $\overline{E}^p_{\bull}$ is the flag (here $\gamma:\mw\to \mw/\mathcal{O}(1)$)
\begin{equation}\label{secondkind}
0\subseteq \gamma(E^p_1)\subseteq \gamma(E^p_2)\subseteq\dots\subseteq \gamma(E^p_{N-1})=\tmw_p.
\end{equation}

We now analyze the (degenerate) enumerative problem in Sections \ref{ceremony1} and \ref{ceremony2} below.

 \subsection{Part of the intersection in $\quot_1$}\label{ceremony1}
Take a coherent subsheaf $\mv\subseteq  \ker \tau=\mw'$ with the corresponding induced (generic) flags. This enumerative problem is the same as counting such subsheaves of
$\mw'$ such that $\mv_p\to \mw_p/E^p_{i^{p}_{a}}$ has kernel of dimension at least $a$ for all $a\in[r+1]$, and $p\in S$. So we need $\mv_p\to  \mw'_p/E^p_{i^{p}_{a}}\cap \mw'_p$ to have  kernel of dimension at least $a$. But $E^p_{i^{p}_{a}}\cap\mw'_p$ has rank $i^{p}_{a}-1$. Let $J^p=\{i^p_1-1<\dots< i^p_N -1\}$ and $E'^p_{\bull}$ the induced flag on $\mw'_p$.
The conditions on $\mv$ are therefore that $\mv_p\to \mw'_p/E'^p_{j^p_{a}}$ has kernel of rank at least $a$. This is just the compactified enumerative problem of subsheaves of the
evenly split bundle $\mw'$ of degree $0$ and rank $r+1$ with Schubert conditions given by $\lambda_i$ at $p_i$:
Note
$$N-(r+1)+a-i^p_{a}=(N-1)-(r+1)+a -(i^p_{a}-1)=\lambda^{(a)}_i, p_i=p.$$
This is the enumerative problem ~\ref{ceremonie1}. Using Section ~\ref{bertram}, we conclude that
the intersection in $\quot_1$ has $\rk \vr$ number of points and that the intersection lies entirely in the open part $A$ (see Lemma \ref{generalposition}, (c)).
\subsection{Part of the intersection in $\quot_2$}\label{ceremony2}
Take a coherent subsheaf $\tmvv\subseteq  \tmw$ with the corresponding induced (generic) flags.
This enumerative problem is the same as counting such subsheaves (of the evenly-split bundle
$\tmw$) so that, setting $\mv$ to be the inverse image of $\tmvv$ under the surjection $\mw\to\tmw$,
$\mv_p\to \mw_p/E^p_{i^{p}_{a}}$ has kernel of dimension at least $a$.

Let $\bar{E}^p_{\bull}$ be the induced flag on $\tmw_p$ and $\bar{E}^p_{N-1}=\tmw_p$. By an elementary calculation,
the kernel of $\tmvv_p\to\tmw_p/\bar{E}^p_{i^p_{a}}$ is isomorphic to the kernel of $\mv_p\to \mw_p/E^p_{i^{p}_{a}}$ for
$a\leq r$. For this, write exact sequences
$$0\to \mathcal{O}(1)_p\to \mv_p\to \tmvv_p \to 0, \mbox{ and } \ 0\to \mathcal{O}(1)_p\to \mw_p/E^p_{i^{p}_{\ell}}\to \tmw_p/\bar{E}^p_{i^p_{\ell}}\to 0.$$
Also note that $\bar{E}^p_{i^p_{a}}$ has rank $i^p_{a}$.

So we want to count subsheaves of $\tmw$ of rank $r$ and degree $-1$ with Schubert conditions given by the data
$\lambda_i$ at $p=p_i$:
note $$(N-1)- r+\ell-i^p_{a}=N-(r+1)+\ell-i^p_{a} =\lambda^{(a)}_i,\ a\leq r, \ \  \ N-(r+1)+(r+1)-i^p_{r+1}=0.$$
Therefore the part of the enumerative problem in  $\quot_2$ is the enumerative problem ~\ref{ceremonie2}, and has $\rk \vl$ points,
and these points lie entirely in $B$ (again using Lemma \ref{generalposition} (c)).

 \subsection{Proof of ~\eqref{morpheusnew}}\label{conservation}
Consider a family of vector bundles $\mw$ over $\pone\times T$  where  $T$ is a smooth curve such that $\mw_t$ is isomorphic to $\mathcal{O}^{\oplus n}$ for
$t\neq t_0$ and isomorphic to $\mathcal{O}(1)\oplus\mathcal{O}(-1)\oplus\mathcal{O}^{\oplus (N-2)}$ for $t=t_0$. Let $\me_{t_0}\in\Fl_S(\mw_{t_0})$ be generic and consider a family of
$\me_t\in \Fl_S(\mw_t)$ specializing to $\me_{t_0}$ and specializing at a fixed point $t_1$ to a general element of $\Fl_S(\mw_{t_1})$ with $t_1\neq t_0$.

We have a relative quot scheme (of quotients of degree $0$ and rank $\ell+1$ of $\mw_t$) $\pi:\widetilde{\quot}\to T$ and the family of solutions to the enumerative problem $i:\mathcal{C}\hookrightarrow\widetilde{\quot}$. We note:
\begin{enumerate}
\item $\pi$ is proper and $i$ is closed.
\item $\pi$ is smooth over $T-\{t_0\}$.
\item Each irreducible component of $\mc$ has dimension at least one. This is because the map $\mc\to T$ is the pull back of an universal $\mc'\to T'$ by a map $T\to T'$.
It turns out that $\mc'$ and $T'$ have the same dimension and hence the assertion follows from standard properties of dimensions of fibers.
\item $\pi$ is smooth at $A\cup B\subseteq \pi^{-1}(t_0)$. $\pi$ is not smooth at $C\subseteq \pi^{-1}(t_0)$.
\item $(\pi\circ i)^{-1}(t_0)\subseteq A\cup B$ and $(\pi\circ i)^{-1}(t_0)$ is a reduced scheme consisting of finitely many points.
\end{enumerate}
By shrinking $T$ if necessary around $t_0$ we may assume that each component of $\mathcal{C}$ surjects onto $T$. Since $(\pi\circ i)^{-1}(t_0)$ is a reduced scheme consisting of finitely many points and each irreducible component of $\mathcal{C}$ has dimension at least one, we see that for $c\in \mathcal{C}$ over $t_0$, $\mathcal{O}_{C,c}$ is a discrete valuation ring (the Zariski tangent space has dimension at most one) . By shrinking $T$ we can assume $\mathcal{C}$ to be smooth and equidimensional of dimension $1$. It is now easy to see that $\mathcal{C}\to T$ is finite and \'etale over a neighborhood of $t_0$. The generic fiber of $\mathcal{C}\to T$ has $\rk\clar$ number of points, and the special fiber over $t_0$ has $\rk\vl +\rk\vr$ number of points. Therefore ~\eqref{morpheusnew} holds.

\subsection{The equality \eqref{morpheusnew} stated in terms of quantum cohomology}\label{leo}
We  return to the setting of Section \ref{setting}. Assume that $\ell$ is the critical level for $\vec{\lambda}$.
The rank $\rk \clar$ is  the coefficient of the class of a point $[\operatorname{pt}]$ in the classical product
\begin{equation}\label{old12}
\sigma_{\lambda_1}\cdot\sigma_{\lambda_2}\cdot\dots\cdot\sigma_{\lambda_n}\in H^*(\Gr(r+1,r+\ell+2)).
\end{equation}

Let $\lambda$ be the $(r+1)\times\ell$ Young diagram $(\ell,0,\dots,0)$. Then $\rk \vr$ (there is a similar expression for $\rk\vl$)  is
the coefficient of $q[\operatorname{pt}]$ in the small quantum product
\begin{equation}\label{new12}
\sigma_{\lambda_1}\star\dots\star\sigma_{\lambda_n}\star\sigma_{\lambda}\in QH^*(\Gr(r+1,r+\ell+1)).
\end{equation}
Note that the Grassmannian appearing in \eqref{new12} is different from the one in \eqref{old12}. In terms of Gromov-Witten numbers (see Remark \ref{notateo}), the sum of $$\langle \sigma_{\lambda_1},\dots,\sigma_{\lambda_n},\sigma_{(\ell,0,\dots,0)}\rangle_1,\text{ and } \langle \sigma_{\lambda^T_1},\dots,\sigma_{\lambda^T_n},  \sigma_{(r,0,\dots,0)}\rangle_1$$ computed for  $\Gr(r+1,r+\ell+1)$ and for $\Gr(\ell+1,r+\ell+1)$ respectively equals
the classical coefficient $\langle \sigma_{\lambda_1},\dots,\sigma_{\lambda_n}\rangle_0$ computed for $\Gr(r+1,r+\ell+2)$.

\begin{remark}
One may ask whether structure constants in the classical cohomology of a $G/P$ (the above is a special case where   $G/P=\Gr(r+1,r+\ell+2)$) similarly decompose as  sums of quantum cohomology structure constants (for possibly different groups). Any classical cohomology structure constant can be interpreted as the enumerative problem of counting suitable reductions of  the structure group of the trivial principal $G$-bundle
on $\pone$ (subject to incidence conditions). We may replace the trivial principal $G$ bundle by its simplest degeneration and then look at the corresponding enumerative problem, and ask if it breaks up into smaller pieces.
\end{remark}
\section{Conformal blocks as generalized theta functions}\label{verlindesection}
\subsection{Notation}

 A quasi-parabolic $\operatorname{SL}_{r+1}$ bundle on $\pone$ is a triple $(\mv,\mf,\gamma)$ where  $\mv$ is a vector bundle on $\pone$ of  rank $r+1$ and deg $0$ with a given trivialization $\gamma:\det\mv\leto{\sim}\mathcal{O}$, and  $\mf=(F^{p_1}_{\bull},\dots, F^{p_n}_{\bull})\in \Fl_S(\mv)$ is a collection of complete flags  on fibers over $p_1,\dots,p_n$. Let $\parbun_{r+1}$ be the stack parameterizing   quasi-parabolic $\operatorname{SL}_{r+1}$ vector bundles on $\pone$.

\subsection{Generalized theta functions}\label{nara}
It is well known that conformal blocks for $\sL_{r+1}$ on a smooth projective curve can be identified with the space of sections, called ``generalized theta functions'', of a suitable line bundle on the moduli space of vector bundles of rank $r+1$ with trivial determinant on that curve (see the survey \cite{Sorger}). A parabolic generalization for $\sL_{r+1}$ was proved in \cite{pauly} which we recall now (but only for $\pone$).

Associated to the data  $\vec{\lambda}=(\lambda_1,\dots,\lambda_n)\in P_{\tilde{\ell}}(\sL_{r+1})$, we can form a line bundle $\mathcal{P}(\sL_{r+1},\tilde{\ell},\vec{\lambda})$ on $\parbun_{r+1}$. The fiber over a point $(\mv, \mf,\gamma)$ is a tensor product
$$D(\mv)^{\tilde{\ell}}\tensor\tensor_{i=1}^n \ml_{\lambda_{i}}(\mv_{p_i},F^{p_i}_{\bull}),$$
where $D(\mv)$ is the determinant of cohomology of $\mv$ i.e., the line $\det H^1(\mathbb{P}^1, \mv)\tensor\det H^0(\mathbb{P}^1, \mv)^*$ and the lines $\ml_{\lambda_i}(\mv_{p_i},F^{p_i}_{\bull})$ are as defined in Section ~\ref{classic}.

It is known that the space of generalized theta functions is canonically identified (up-to scalars) with the dual of the space of conformal blocks \cite{pauly}: Let $\xpp=(\pone,p_1,\dots,p_n)\in \operatorname{M}_{0,n}$.
\begin{equation}\label{verlinde}
 H^0(\parbun_{r+1},  \mathcal{P}(\sL_{r+1},\tilde{\ell},\vec{\lambda}))\leto{\sim} (\Bbb{V}_{\sL_{r+1},\vec{\lambda},\tilde{\ell}})_{\xpp}^*.
\end{equation}
The determinant of cohomology $D(\mv)$ of a vector bundle with a given trivialization $\gamma:\det\mv\leto{\sim}\mathcal{O}$ carries a canonical element $\theta$. To construct $\theta$, let  $P\in\pone$ and consider $0\to\mv(-P)\to\mv\to\mv_P\to 0$
which sets up an isomorphism $D(\mv)=D(\mv(-P))$ (note that $\det \mv_P$ has a given trivialization $\gamma_P$). But $\chi(\pone,\mv(-P))=0$ and hence $D(\mv(-P))$ carries a canonical theta section. Hence one gets an element $\theta$ of $D(\mv)$. This element $\theta$ does not depend upon $P$ because otherwise we will get non-trivial functions in $P\in \pone$. We  may also apply the isomorphism \eqref{verlinde} at level $1$ with vacuum representations at $p_i$ (i.e., $\lambda_i=0$) to construct $\theta$. 

The theta section of $D(\mathcal{V}(-P))$ vanishes along the locus where $\dim H^0(\mathcal{V}(-P))\neq 0$. Since $\mv$ is of degree $0$, this is same as saying that $\theta$ vanishes at $\mv$ iff $\mv$ is non-trivial as a vector bundle. 

Therefore  one has a canonical section (also denoted $\theta$) of the determinant of cohomology bundle on the moduli stack of vector bundles on $\pone$ with trivialized determinant. Note that we may normalize $\theta$ by requiring it to be the canonical one when $\mv=\mathcal{O}^{\oplus r+1}$ (although this is not needed, it is true that the $\theta$ as constructed is already normalized).

Multiplication by $\theta\in D(\mv)$ sets up an injective map
\begin{equation}\label{multiply}
H^0(\parbun_{r+1},  \mathcal{P}( \sL_{r+1},\tilde{\ell},\vec{\lambda}))\to H^0(\parbun_{r+1},  \mathcal{P}( \sL_{r+1},\tilde{\ell}+1,\vec{\lambda})).
\end{equation}

\begin{remark}\label{vanishing1}
The image of (\ref{multiply}) necessarily vanishes on points $(\mv, \mf,\gamma)\in\parbun_{r+1}$ with $\mv$ not isomorphic to $\mathcal{O}^{\oplus r+1}$.
\end{remark}

Note that $\parbun_{ r+1}$ has a classical part,  $\parbun_{ r+1}^c$, the open substack  where the underlying vector bundle is trivial. It is easy to see that
$$H^0(\parbun_{ r+1}^c, \mathcal{P}(\sL_{ r+1},\tilde{\ell},\vec{\lambda}))=  {A}_{ r+1,\vec{\lambda}}^*$$
via the map $\pi:\Fl(V)^n\to\parbun_{ r+1}^c$, where $V$ is a vector space of dimension $ r+1$ with trivialized determinant, which sets up  $\parbun_{ r+1}^c$ as a stack quotient
$\Fl(V)^n/\operatorname{SL}(V)$ ( $\pi$ pulls back $\mathcal{P}(\sL_{ r+1},\tilde{\ell},\vec{\lambda})$ to $\tensor_{i=1}^n\ml_{\lambda_i}$). We therefore obtain injective maps
\begin{equation}\label{duall}
H^0(\parbun_{ r+1},  \mathcal{P}( \sL_{r+1},\tilde{\ell},\vec{\lambda}))\to  {A}_{ r+1,\vec{\lambda}}^*,
\end{equation}
which are compatible with ~\eqref{multiply} (because the canonical section of $D(\mv)$ is $1$ on $\Fl(V)^n$ with our normalization). Note that if the map \eqref{duall} is an isomorphism for $\tilde{\ell}$ then it is also an isomorphism for $\tilde{\ell}+1$.

As a final compatibility (given ~\eqref{verlinde}), we note that ~\eqref{duall} is dual to the canonical surjection (up to scalars)
$${A}_{ r+1,\vec{\lambda}}\to (\Bbb{V}_{\sL_{ r+1},\vec{\lambda},\tilde{\ell}})|_{\xpp}.$$

\section{Proof of Theorem \ref{main}}\label{books4}

 To prove that the composite
 \begin{equation}\label{compo}
\vr^*\to\clar^* \to\clal\to \vl,
\end{equation}
is the zero map (see Section \ref{sketch}, Step 2), we use the enumerative interpretations of $\clar$ and $\clal$. In the setting of Sections ~\ref{books1} and \ref{Schubert}, the duality maps are explicit for sections that are defined by the enumerative problem of classical Schubert calculus. To get explicit representatives for the images of $\vr^*$ in $\clar^*$, we use the (degenerate) enumerative problem of Section \ref{ceremonie}. We will show that there are natural isomorphisms (assuming $\ell$ is the critical level for $\vec{\lambda}$)
$$H^0(\parbun_{r+1},  \mathcal{P}( \sL_{r+1},{\ell+1},\vec{\lambda}))\leto{\sim}  {A}_{r+1,\vec{\lambda}}^*$$
 $$H^0(\parbun_{\ell+1},  \mathcal{P}( \sL_{\ell+1},{r+1},\vec{\lambda}^T))\leto{\sim}  {A}_{\ell+1,\vec{\lambda}^T}^*.$$
 By Section \ref{WittenEnumerative}, we already know the equality of ranks in the above isomorphisms. The resulting duality map
 \begin{equation}\label{monday}
H^0(\parbun_{r+1},  \mathcal{P}( \sL_{r+1},{\ell+1},\vec{\lambda}))\leto{\sim} H^0(\parbun_{\ell+1},  \mathcal{P}( \sL_{\ell+1},{r+1},\vec{\lambda}^T))^*
\end{equation}
 can be controlled in terms of a basis of sections coming from the enumerative problem from Section \ref{ceremonie}. This and Remark \ref{vanishing1} allow us to  chase elements through the composition \eqref{compo}.
 \begin{remark}\label{mondayremark}
 The isomorphism \eqref{monday} is an example of a ``parabolic strange duality'' isomorphism \cite{RO}. It reads via \eqref{verlinde}, for $\sum |\lambda_i|=(r+1)(\ell+1)$, as
 $$\vrplusone|_{\xpp}^*\leto{\sim}\vlplusone|_{\xpp}.$$
 The above isomorphism is identified with the isomorphism \eqref{duality}. It is important for our considerations to know that \eqref{duality} is independent of any choices (of $\xpp=(p_1,\dots,p_n)$).
 \end{remark}
\subsection{Enlargement of the duality divisor D}\label{enlarge}
Let $\mm=\parbun_{r+1},\mnn=\parbun_{\ell+1}$, $\mm^c=\parbun_{r+1}^c$ and $\mnn^c=\parbun_{\ell+1}^c$. There are natural maps $\mm^c\to\mm$ and $\mathcal{N}^c\to\mathcal{N}$.

 We will define a divisor in $D\subseteq \mm\times\mathcal{N}$ extending the divisor on $\mm^c\times\mnn^c$ of Section \ref{classic}. A point $(\mv,\mf,\mt,\widetilde{\mathcal{G}})$ is in $D$ iff there is a non-zero map
$\phi:\mv\to \mq$ so that for $a\in [r+1]$ and $i\in [n]$,
\begin{equation}\label{defD2}
\phi(F^{p_i}_{a})\subseteq G^{p_i}_{\ell+1-\lambda_i^{(a)}}.
\end{equation}
where $\mq=\mt^*$ and $\mgg\in\Fl(\mq)$ the flags induced from $\widetilde{\mgg}\in\Fl(\mt)$.

We can recast this into a ``determinantal scheme" as follows: define a locally free sheaf $\mathcal{K}$ on $\mathcal{M}\times\mathcal{N}\times\pone$ as follows. The fiber of
$\mathcal{K}$ at a point $b=(\mv,\mf,\mt,\widetilde{\mathcal{G}})$ is  given by (as a  bundle on $\pone$)
$$0\to\mathcal{K}_b\to \underline{\Hom}(\mv,\mq)\to \oplus_{i=1}^n i_{p_i,*}\frac{\Hom(\mv_{p_i},\mq_{p_i})}{\mathcal{P}_{\lambda_i}(F^{p_i}_{\bull}, G^{p_i}_{\bull})}\to 0.$$

With the conditions that we have assumed, $\chi(\pone,\mathcal{K}_b)=0$, so there is a canonical element $s\in D(\mathcal{K}_b)$ which vanishes iff $x\in D$. The determinant of cohomology $D(\mathcal{K}_b)$ is  given by
$$\tensor_{i=1}^n(\ml_{\lambda_{i}}(F^{p_i}_{\bull})\tensor \ml_{\lambda^T_{j}}(\widetilde{G}^{p_i}_{\bull}))\tensor D(\mv^*\tensor\mt^*).$$ We can rewrite the above as the following:
$$(\tensor_{i=1}^n(\ml_{\lambda_{i}}(F^{p_i}_{\bull}))\tensor D(\mv^*)^{\ell+1})\boxtimes (\tensor_{i=1}^n \ml_{\lambda^T_{i}}(\widetilde{G}^{p_i}_{\bull})\tensor D(\mt^*)^{ r+1})=\ma_{\mv,\mf}\boxtimes \mb_{\mt,\widetilde{\mg}}$$
where we have introduced the notation $\ma=\mathcal{P}( \sL_{r+1},\ell+1,\vec{\lambda})$ and $\mb=\mathcal{P}( \sL_{\ell+1},r+1,\vec{\lambda}^T)$, used the canonical isomorphisms $D(\mv)=D(\mv^*)$ and $D(\mt)=D(\mt^*)$ (the determinants of $\mv$ and $\mt$ are trivialized), and the following lemma
(see e.g., \cite{Fa}, Theorem I.1):
\begin{lemma}Let $\mv$, $\mw$ be vector bundles on $\pone$.
Suppose  $\deg \mv=\deg\mw=0$.  There is a natural isomorphism
$D(\mv\tensor\mw)\to D(\mv)^{\rk \mw}\tensor D(\mw)^{\rk \mv}$
which specializes to the obvious one when $\mv=\mathcal{O}^{\oplus \rk V}$ (note that $D(\mathcal{O})$ is canonically trivial).
\end{lemma}
Let $\ma^c$ (resp. $\mb^c$) be the pull back of $\ma$ (resp. $\mb$)  under the map $\mm^c\to \mm$ (resp. $\mathcal{N}^c\to\mathcal{N}$). The duality divisor $D$ of $\mm\times \mnn$ pulls back to the duality divisor of Section \ref{classic} in $\mm^c\times\mathcal{N}^c$.

We introduce the following notation:
$$\ma_{\ell}=\mathcal{P}(\sL_{r+1},{\ell},\vec{\lambda}), \ \  \ma_{r}=\mathcal{P}( \sL_{\ell+1},{r},\vec{\lambda}^T).$$

To set the stage for the final verification, note that (comparing ranks and the enumerative interpretation of Section \ref{bb3}), the maps 
\begin{equation}\label{stage}
H^0(\mm,\ma)\leto{\sim} H^0(\mm^c,\ma^c)\leto{\sim} {A}_{r+1,\vec{\lambda}}^*
\end{equation}
are isomorphisms. There are also natural maps (see ~\eqref{multiply})
$H^0(\mm,\ma_{\ell})\to H^0(\mm,\ma)$.
These maps are compatible with ~\eqref{stage} (see \eqref{duall}).
Let us  also note that each $y\in\mathcal{N}$ gives us an element ``evaluation at $y$'' $\eva_y\in H^0(\mathcal{N}, \mb)^*$ which is well defined up-to scalars.
\begin{remark}\label{vanishing2}
By Remark ~\ref{vanishing1}, it follows that the image of $\eva_y$ vanishes in $H^0(\mathcal{N}, \mb_r)^*$ if the underlying bundle of $y$ is non-trivial.
\end{remark}

\subsection{Geometric data}\label{geometricdatum}
Recall the enumerative problem of Section ~\ref{ceremonie} corresponding to a general point $\me\in\Fl_S(\mw)$ with $\mw=\mathcal{O}(1)\oplus\mathcal{O}(-1)\oplus\mathcal{O}^{\oplus N-2}$.
Let $\cone=\rk\vr$, $\ctwo=\rk\vl$, so $\cone+\ctwo=\rk\clar$. Solutions to the enumerative problem in $A\subseteq \quot$  produce vector bundles on $\pone$ equipped with complete (induced) flags
$(\mv_1,\mf_1),\dots,(\mv_{\cone},\mf_{\cone})$,  where $\mf_i\in\Fl_S(\mv_i)$. We  obtain  quotients equipped with complete (induced) flags  $(\mq_1,\mgg_1),\dots,(\mq_{\cone},\mf_{\cone})$ where $\mq_i=\mw/\mv_i$ and $\mgg_i\in\Fl_S(\mq_i)$. Note that the bundles $\mq_i$ are non-trivial since they have $\mathcal{O}(-1)$ as a direct summand (see Section \ref{proofofceremonie}). 

Similarly, solutions to the enumerative problem in  $B\subseteq \quot$ produce  bundles and quotients (each equipped with flags at points of $S$): $(\mv'_1,\mf') \dots,(\mv'_{1},\mf'_{\ctwo})$ and
$(\mq'_1,\mgg_1),\dots,(\mq'_{\ctwo},\mgg'_{\ctwo})$ (again $\mv'_a\subseteq \mw$ and $\mq'_a=\mw/\mv'_a$). The bundles $\mv_a$ are non-trivial since they have $\mathcal{O}(1)$ as a subsheaf (see Section \ref{proofofceremonie}).

Let $(\mt_i,\widetilde{\mgg}_i)$ (resp. $(\mt'_a,\widetilde{\mgg}_a$)) be the duals of $(\mq_i,{\mgg}_i)$ (and resp. $(\mq_a,{\mgg}_a)$), $i\in[\cone]$ (resp. $a\in[\ctwo]$). The bundles $\mv_i,\mt_i,\mv'_a,\mt'_a, i\in[\cone],a\in[\ctwo]$ are of degree zero. We choose and fix trivializations of their determinants.  We therefore obtain
\begin{enumerate}
\item Points $x_1,\dots, x_{\cone};\ x'_1,\dots,x'_{\ctwo}\in\mm$.  Here, $x_i$ (resp. $x'_a$) are the points $(\mv_i,\mf_i)$ (resp. $(\mv'_a,\mf'_a)$).
\item Points $y_1,\dots, y_{\cone};\ y'_1,\dots,y'_{\ctwo}\in\mathcal{N}$.  Here, $y_i$ (resp. $y'_a$) are the points $(\mt_i,\widetilde{\mgg}_i)$ (resp. $(\mt'_a,\widetilde{\mgg}'_a)$).
\item $s(x_i,y'_a)=0$ and $s(x'_a,y_i)=0$ for all $i\in[\cone]$ and $a\in[\ctwo]$.  This is because there are natural non-zero maps $\mv_i\to \mq'_a$ ``inclusion in $\mw$ followed by projection" (resp. $\mv'_a\to\mq_j$) which satisfy the conditions of ~\eqref{defD2} (with the corresponding flags $\mf_i$ and $\mgg'_a$ (resp. $\mf'_a$ and $\mgg_i$)).
\item $s(x'_a,y'_b)\neq 0$ iff $a=b$, and $s(x_i,y_j)\neq 0$ iff $i=j$, $i,j\in[\cone]$ and  $a,b\in[\ctwo]$. This is a consequence of transversality  in the enumerative problem (as in remarks \ref{transversality} and Lemma \ref{transverseo}, (b)).
\item By Lemma \ref{books},  $\delta_i=s(.,y_i)$ and $\gamma_a= s(., y'_a)$, $i\in[\cone]$, $a\in[\ctwo]$,  form a basis for $H^0(\mm,\ma)$.
\end{enumerate}
The desired vanishing of $\vr^*\to \vl$ 
 follows from
\begin{proposition} The composition of the following maps is zero:
$$H^0(\mm,  \ma_{\ell})\to H^0(\mm, \ma)\to H^0(\mathcal{N}, \mb)^*\to H^0(\mathcal{N}, \mb_r)^*.$$
\end{proposition}
\begin{proof}
The image of a section $\tau\in H^0(\mm,\ma_{\ell})$ in $H^0(\mm,\ma)$ can be expressed as a linear combination of $\delta_i$'s and $\gamma_a$'s. Now evaluate such an
expression at the points $x'_a$. By Remark ~\ref{vanishing1}, $\tau(x'_a)=0$ because the vector bundle $\mv_a$ underlying $x'_a$ is non-trivial. Thus $\tau$ is a linear combination of the sections $\delta_i$ by (iii) and (iv). By Lemma ~\ref{books}, the duality map
$H^0(\mm, \ma)\to H^0(\mathcal{N}, \mb)^*$  sends $\delta_i$ to $\eva_{y_i}$. Remark ~\ref{vanishing2} gives that the image of $\eva_{y_i}$ in $H^0(\mathcal{N}, \mb_r)^*$
is zero (because the  vector bundle $\mq_i^*$ underlying $y_i$ is non-trivial). Hence the image of $\tau$ in $H^0(\mathcal{N}, \mb_r)^*$ is zero as desired.
\end{proof}
\begin{remark}\label{finale}
The points $x'_a$ in $\mm$ have non-trivial underlying bundles, and are yet parabolic-semistable for the linearization
$\mathcal{P}( \sL_{r+1},{\ell+1},\vec{\lambda})$. This is because of the existence of parabolic bundles $y'_a$ and the non-existence of non-zero maps of parabolic bundles from $x'_a$ to $y'_a$ (using (4) above and a parabolic generalization of a  method of Faltings, see \cite[A.1]{Belkale1}).
\end{remark}
\begin{remark}\label{comparison_new}
 \begin{itemize}
 \item  The proof of the vanishing statement, Proposition \ref{newthing} even for $\sL_2$ (noted in
 \cite{Fakh}) uses some representation theory of $\sL_2$, and does not use explicit formulas \cite{Sorger}. We are not able to even see, using formulas, that for $\ell$ large enough and fixed $\vec{\lambda}$,  $c_1(\vr)=0$.  
  \item Formula (a) in Proposition \ref{strange} also does not seem to follow from explicit rank formulas \cite{Sorger} (even $n=3$). For example, the ``denominators'' in the explicit formulas are $(r+1+\ell)$ on one side and $(r+1+\ell+1)$ on the other (in strange duality settings where such equalities follow from
 formulas, the denominators are the same).
 \item Assuming Proposition \ref{newthing}, and (a) in Proposition \ref{strange}, one could ask if (b) follows from Fakhruddin's Chern class formulas. 
  The factorization data required (in the formulas of \cite{Fakh}) for $\vr$ and $\vl$ are not known to be related, even for $n=4$ (in particular, one cannot assume that the factorization data is again at critical level), so it is not clear how to proceed with a formulaic approach.
  \end{itemize}
\end{remark}
\section{Compactifications and Gromov-Witten numbers}\label{bertram} In Section \ref{outline}, we consider intersections in  the entire  Quot scheme instead of an open subset as done in Gromov-Witten theory. The goal of this section is to show that compactifying parameter spaces does not change the enumerative numbers. The main result of this section is Proposition \ref{bertie}, a generalization of a result of Bertram  \cite{Bertram}.

 Let $\mw$ be a  vector bundle on $\pone$ of rank $N$ and degree $-D$. Let $d$ and $\tr$ be integers with $N>\tr> 0$ (and $d,\ D$  possibly
negative). Define {$\mfm(d,\tr,\mw)$} to be the moduli space of
sub bundles of  $\mw$ of degree $-d$ and rank $\tr$. This is an open
subset of the quot scheme $\quot(d,\tr,\mw)$ of quotients of $\mw$ of degree $d-D$ and
rank $N-\tr$.
\begin{proposition}\label{bigoneprop}
Suppose $\mw$ is evenly-split
\begin{enumerate}
\item $\quot(d,\tr,\mw)$  a smooth projective variety (possibly empty) of dimension
$\tr(N-\tr)+dN- D\tr$.
\item  $\mfm(d,\tr,\mw)\subseteq\quot(d,\tr,\mw)$ is dense and connected.
\item The subset  $\mfm^o(d,\tr,\mw)$ of
$\mfm(d,\tr,\mw)$ consisting of
 evenly split sub bundles $\mv\subseteq\mw$ such that $\mw/\mv$ is also
 evenly split, is open and dense in $\mfm(d,\tr,\mw)$.
 \end{enumerate}
\end{proposition}
See  Section 12.1 of ~\cite{b4} for a proof of the above (certainly well known).

Consider a $5$-tuple $(d,\tr,D,N,I)$ where $d,D,\tr,N$ are as above and $I$  assigns to each
$p\in S$ an element $I^p\in \binom{[N]}{\tr}$. We will use the notation
$I^{p}=\{i^p_1<\dots<i^p_{\tr}\}$ for $p\in \mpp$,
also introduce Young diagrams $\lambda_p$
by the formula $\lambda_{p}^{(a)}=N-\tr+a-i^p_a$ for $a=1,\dots,\tr.$
Assume the numerical condition (which leads to enumerative problems with finitely many solutions)
\begin{equation}\label{dimii}
\tr(N-\tr) -D\tr +dN=\sum_{p\in S}|\lambda_p|.
\end{equation}

Let $\mw$ be a vector bundle of degree $-D$ and rank $N$. Let $\me\in\Fl_S(\mw)$. We fix the data
of $\me$ and $\mw$ in this section.

Form the natural morphism
$$\pi:\mfm(d,\tr,\mw)\to \prod_{i=1}^n \Gr(\tr,\mw_{p_i})$$
and consider
 \begin{equation}\label{defomega}
 \Omega=\pi^{-1}(\prod_{i=1}^n \Omega_{I^{p_i}}(E^{p_i}_{\bull})).
 \end{equation}
 
 Since $\prod_{i=1}^n \Gr(\tr,\mw_{p_i})$ is a homogeneous space, Kleiman's transversality implies:
 \begin{lemma}\label{generalposition}
 Suppose $\mw$ is evenly split and  $\me\in\Fl_S(\mw)$ is general, then,
 \begin{enumerate}
 \item[(a)]$\Omega$ is a finite set.
 \item[(b)] $\Omega$ coincides with the intersection of inverse images of open Schubert cells $\pi^{-1}(\prod_{i=1}^n \Omega^o_{I^{p_i}}(E^{p_i}_{\bull}))$, and is therefore a transverse intersection.
 \item[(c)] $\Omega\subseteq \mfm^o(d,\tr,\mw)$ ($\mfm^o(d,\tr,\mw)$ was defined in Proposition \ref{bigoneprop}).
 \end{enumerate}
\end{lemma}
 
 If $\mw$ is not general, we have a weaker result by replacing $\mfm(d,\tr,\mw)$  by its smooth open (possibly empty) subset $\mfm^{sm}(d,\tr,\mw)$
 of points where the tangent space is of the expected dimension (i.e., $\mv\subseteq \mw$ such that  $\operatorname{Ext}^1(\mv,\mw/\mv)=0$). The following lemma was used in Section \ref{geometricdatum}.
 \begin{lemma}\label{transverseo}
 Suppose $\me\in\Fl_S(\mw)$ is general. 
 \begin{enumerate}
 \item[(a)] The intersection $\Omega\cap \mfm^{sm}(d,\tr,\mw)$ is a smooth and transverse intersection in $\mfm^{sm}(d,\tr,\mw)$ which coincides with $\mfm^{sm}(d,\tr,\mw)\cap \pi^{-1}(\prod_{i=1}^n \Omega^o_{I^{p_i}}(E^{p_i}_{\bull}))$.
 \item[(b)] 
Suppose $\mv\in \Omega\cap \mfm^{sm}(d,\tr,\mw)$. Let $\mq=\mw/\mv$, and $\mf\in\Fl_S(\mv)$, $\mgg\in \Fl_S(\mq)$ the induced flags. 
 There are no non-zero maps $\phi:\mv\to \mq$ such that for $a=1,\dots,r+1$ and $i=1,\dots, n$, the inclusion \eqref{defD2} holds. 
 \end{enumerate}
 \end{lemma}
 \begin{proof}
 Part (a) follows from Kleiman's transversality by restricting the domain of $\pi$ to the smooth open subset $\mfm^{sm}(d,\tr,\mw)$.
 For (b), note that  the tangent space to $\mfm(d,\tr,\mw)$ at $\mv$ is $\operatorname{Hom}(\mv,\mq)$, and the tangent spaces to Schubert cells are as described in Section \ref{added2}. Now, (b) follows from the transversality  and zero dimensionality of the intersection $\Omega$ at $\mv$.
\end{proof}

 It is important to work with compact parameter spaces while degenerating enumerative problems.  The set $\mfm(d,\tr,\mw)$
is not compact, but $\quot(d,\tr,\mw)$ is  a projective variety which contains $\mfm(d,\tr,\mw)$ as an open subset. We therefore consider the following modified enumerative problem:  
\begin{itemize}
\item Let $\Omega'\subseteq \quot(d,\tr,\mw)$ be the set of subsheaves $\mv\subseteq \mw$ of degree $-d$ and rank $\tr$ 
such that for all $p\in S,a\in[\tr]$, the map $\mv_p\to \mw_p/E_{i^{p}_{a}}$ has kernel of dimension at least $a$.  
\end{itemize}
The following is due to Bertram (Lemma 2.2A in \cite{Bertram}), when $D=0$. We give an obvious extension of Bertram's proof for arbitrary $D$ in Section \ref{bertieproof}.
\begin{proposition}\label{bertie}\cite{Bertram}
Suppose that $\mw$ is evenly split and $\me\in\Fl_S(\mw)$ is general. Then, 
$\Omega=\Omega'$.
\end{proposition}

\subsection{Proof of Proposition \ref{bertie}}\label{bertieproof}
 We reproduce a variant of Bertram's proof, adapted to our notation. We stratify $\quot(d,\tr,\mw)$ with strata $\mfm(\epsilon)$ parameterized by functions $\epsilon:S\to \Bbb{Z}$ such that $0\leq\epsilon(p)\leq \tr$. We define $\mfm(\epsilon)$ to be the (quot) scheme of subsheaves
$\mv\subseteq \mw$ such that the map $\mv_p\to\mw_p$ has kernel $K_p$ of rank $\epsilon(p),\ p\in S$. Let $\mv$ be a generic point in  $\mfm(\epsilon)\cap\Omega'$. Suppose the subset $\mv_p/K_p$ is a subspace of $\mw_p$ in the Schubert cell of $\operatorname{Gr}(\tr-\epsilon(p),\mw_p)$ corresponding to the $(\tr-\epsilon(p))$-tuple $J^p=\{j^p_1<\dots<j^p_{\tr-\epsilon(p)}\}$. We clearly need
$j^p_{a}+\epsilon(p)\leq i^p_{a +\epsilon(p)},\ a\in[\tr-\epsilon(p)]$.

One therefore gets that the codimension of the Schubert cell corresponding to $J^p$ is at least the codimension of the Schubert cell corresponding to $I^p$ minus
$\sum_{a=1}^{\epsilon(p)} (N-\tr+a-i^p_a)$.  The  dimension of each irreducible component  of  $\mfm(\epsilon)\cap \Omega'$ is therefore (by a simple calculation, using \eqref{dimii} and Kleiman transversality) at most
\begin{equation}\label{intersectione}
\dim\mfm(\epsilon)-\dim\mfm(d,\tr,\mw)  + \sum_{p\in S}\sum_{a=1}^{\epsilon(p)} (N-\tr+a-i^p_a).
\end{equation}

\begin{lemma}\label{zen} Suppose $\mfm(\epsilon)\neq\emptyset$. Then, $\Gr(\epsilon)$ is smooth and connected, and
$$\dim\mfm(\epsilon)= \dim\mfm(d,\tr,\mw)-N\sum_{p\in S}\epsilon(p) + \sum_{p\in S} \epsilon(p)(\tr-\epsilon(p)).$$
\end{lemma}

Given the lemma, we see that the quantity ~\eqref{intersectione}, is less than or equal to
$\sum_{p\in S}\sum_{a=1}^{\epsilon(p)} (a-i^p_a -\epsilon(p))<0$
unless $\epsilon(p)=0$ for each $p$. Therefore the enumerative intersection takes place over the stratum corresponding to $\epsilon(p)=0$. This stratum has
an open dense subset given by $\mfm(d,\tr,\mw)$. By Kleiman's transversality, the intersection occurs entirely on this subset.
\subsubsection{Proof of Lemma ~\ref{zen}}
Consider the variety $\mathcal{A}$ of tuples $(K_{p_1},\dots,K_{p_n},\mv)$ where $\mv$ is a coherent subsheaf of $\mw$ of degree $-d$ and rank $r$, $n$ is the number of points in $S$ and $K_p\subset\mv_p$ are $\epsilon(p)$
dimensional subspaces such that the composites $K_p\to\mv_p\to\mw_p$ are zero for each $p$. Our $\mfm(\epsilon)$ is an open subset of $\mathcal{A}$ and so it suffices to find the dimension of $\mathcal{A}$.

Given a datum $(K_{p_1},\dots,K_{p_n},\mv)$, we can define a coherent sheaf $\widetilde{\mv}$ containing $\mv$ as follows:
\begin{itemize}
\item $\widetilde{\mv}$ coincides with $\mv$ outside of $S$,
\item In a neighborhood $U_p$ of $p\in S$, sections of $\widetilde{\mv}$ are determined as follows. Let $t_p$ be a uniformizing parameter at $p$. Sections of $\widetilde{\mv}$ in $U_p$ are meromorphic sections $s$ of $\mv$ such that $t_p s$ is a regular section of $\mv$ on $U$, with  fiber at $p$ in $K_p$. 
\end{itemize}
Clearly $\widetilde{\mv}$ is a coherent subsheaf of $\mw$.
We also obtain subspaces $Q_p=\im(\mv_p\to\widetilde{\mv}_p)\subseteq \widetilde{\mv}_p$ of rank $\tr-\epsilon(p)$ for all $p\in S$. Therefore one obtains a tuple
$(Q_{p_1},\dots,Q_{p_n},\widetilde{\mv})$. Let $\mb$ be the scheme of such tuples. It is easy to see that
$(K_{p_1},\dots,K_{p_n},\mv)$ can be recovered from $(Q_{p_1},\dots,Q_{p_n},\widetilde{\mv})$ (and vice-versa, see Lemma A.2 in \cite{b4}).
The schemes $\ma$ and $\mb$ are therefore isomorphic. We calculate the dimension of $\mb$:  The dimension of the space of such $\widetilde{\mv}$ is $\dim\mfm(d,r,\mw)-N\sum_{p\in S}\epsilon(p)$. Adding the dimensions of the Grassmann bundles of $Q_p$, we see that Lemma \ref{zen} holds.

\section{Involution Identities} \label{Weyl}
There is an involution on $P_{\ell}(\frg)$ which sends a weight $\lambda$ to the highest weight $\lambda^*=-w_0(\lambda)$ of  $(V_{\lambda})^*$, where $w_0$ is the longest element of the Weyl group of $\frg$. For example if $\frg=\sL_{r+1}$ is of type $A$, then $\lambda^*$ is given by $\lambda^c$ for any  weight $\lambda$ represented by a Young diagram of size $(r+1)\times \ell$.
\begin{proposition}\label{relation}
On $\overline{\operatorname{M}}_{0,\op{n}}$, $\mathbb{D}_{\frg,\vec{\lambda},\ell}=\mathbb{D}_{\frg,\vec{\lambda}^*,\ell},$  where $\vec{\lambda} \in P_{\ell}(\frg)^n$.
\end{proposition}

By results in ~\cite{Fakh}, Proposition \ref{relation} reduces to the case $n=4$:

\begin{lemma}\label{relation1} On $\overline{\operatorname{M}}_{0,4}$,
$\deg\mathbb{V}_{\frg,\vec{\mu},\ell}=\deg\mathbb{V}_{\frg,\vec{\mu^*},\ell}.$
\end{lemma}
\begin{proof}(of Lemma \ref{relation1}) Let $\mu$ be any dominant integral weight of $\frg$ and $\mu^*=-w_0(\mu)$. The Cartan Killing form is invariant under the action of the Weyl group of $\frg$, so
$(\mu, \mu+2\rho)=(\mu^*, \mu^*+2\rho)$, where $\rho$ is the half sum of positive roots.
We conclude the proof using Corollary 3.5 in ~\cite{Fakh}, and the above.
\end{proof}
The following lemma relates the critical and theta levels for $\sL_{r+1}$. Here, $\vec{\lambda}$ is an $n$-tuple of  {normalized} integral weights for $\sL_{r+1}$ such that  $r+1$ divides $\sum_{i=1}^n|\lambda_i|$. One also obtains a different proof, only for $\frg=\sL_{r+1}$, of theta level vanishing  (Remark \ref{ThetaVanishing}); see (b) below.
\begin{lemma}\label{averagenew}
\begin{enumerate}
\item[(a)] $\theta(\sL_{r+1},\vec{\lambda})=\frac{1}{2}\cdot (c(\sL_{r+1},\vec{\lambda})+c(\sL_{r+1},\vec{\lambda}^*))$. 
\item[(b)] If $\ell > \lL(\sL_{r+1},\vec{\lambda})$,  then $\mathbb{D}_{\sL_{r+1},\vec{\lambda},\ell}=c_1(\mathbb{V}_{\sL_{r+1},\vec{\lambda},\ell})=0$.
\item[(c)] If the weights $\lambda_1^*,\dots, \lambda_n^*$ are a permutation of the weights $\lambda_1,\dots,\lambda_n$, then $c(\sL_{r+1},\vec{\lambda})=\theta(\sL_{r+1},\vec{\lambda})$.
\end{enumerate}
\end{lemma}
\begin{proof}
To prove (a), normalize $\lambda_i^*$ to obtain weights $\mu_i$. It is easy to see that
$$|\mu_i|=(r+1)\ell- |\lambda_i|- (\ell- \lambda_i^{(1)})(r+1)= (r+1)\lambda_i^{(1)}-|\lambda_i|, $$
which yields (a). It is easy to see that (b) follows from (a), and Propositions \ref{relation} and \ref{newthing} (applied to $\vec{\lambda}$ and $\vec{\lambda}^*$). Finally,  (c) follows from (a). 
\end{proof}

\section{Strange duality and  critical level symmetries}\label{strangeandcritical}

Several rank-level (or ``strange'')  dualities have been proposed by many authors (inspired by work in mathematical physics e.g., \cite{NakTsu,Belkale1,MO2,RO}), over smooth pointed curves.  In genus $0$ these take the form of isomorphisms (well defined up to scalars) over $\operatorname{M}_{0,n}$ of the form, for suitable conformal embeddings $\frg_1 \oplus\frg_2\subseteq \frg$ , and $x=(\mathbb{P}^1,p_1,\dots,p_n) \in \operatorname{M}_{0,n}$
\begin{equation}\label{SDgeneral}
\mathbb{V}_{\frg_1, \vec{\lambda}, \ell_1}|_{x}\cong\mathbb{V}_{\frg_2, \vec{\mu}, \ell_2}|_{x}^*.
\end{equation}
These projective isomorphisms are defined only on the interior $\operatorname{M}_{0,n}$. For example, there is an isomorphism on $\operatorname{M}_{0,n}$ \cite{NakTsu,RO}, well defined up to scalars, of the form ($\vec{\lambda}\leftrightarrow \vec{\mu}$ is essentially the transpose up to diagram automorphisms)
\begin{equation}\label{SD}
\mathbb{V}_{\sL_{r+1}, \vec{\lambda}, \ell}|_{x}\cong \mathbb{V}_{\sL_{\ell}, \vec{\mu}, r+1}|_{x}^*.
\end{equation}

While our critical level symmetries are reminiscent of this ``strange duality",  the bundles $\mathbb{V}_{\sL_{r+1},\vec{\lambda}, \ell}$ and  $\mathbb{V}_{\sL_{\ell+1}, \vec{\lambda}^T, r}$ on $\overline{\operatorname{M}}_{0,\op{n}}$ are not dual, and often have different ranks.  Because they are both globally generated, their first Chern classes are base point free, and hence effective.   For two effective divisors to be dual, would mean they are trivial.

Moreover, the classical duality \eqref{duality} can (eventually) be viewed as a special case of the general strange duality \eqref{SD} given above (see Remark \ref{mondayremark}).
We could say that the critical level identities of this paper are orthogonality relations ``via" the strange duality \eqref{SD}. So whenever both sides of \eqref{SDgeneral} are classical, i.e., coincide with co-invariants, there is perhaps a symmetry of conformal blocks divisors (at levels $\ell_1-1$, $\ell_2-1$ for $\frg_1$ and $\frg_2$).

However, in \cite{SM}, the third author studies identities that do come from standard  level-rank dualities associated to conformal embeddings of affine Lie algebras. Very roughly speaking, a conformal embedding of affine Lie algebras $\frg_1 \oplus\frg_2\subseteq \frg$  producing a level rank duality,
gives rise to a corresponding relation of conformal blocks divisors which takes the shape:
 $$\mathbb{D}_{\frg_1, \vec{\lambda}, \ell_1}+\mathbb{D}_{\frg_2, \vec{\mu}, \ell_2}= c\cdot \mathbb{D}_{\frg, \vec{\nu}, 1} + E,$$
where $E$ is a
sum of boundary and $\psi$ classes and $c\neq 0$, determined by the embedding.

\bibliographystyle{alphaurl}

\begin{thebibliography}{{Muk}13}

\bibitem[Agn95]{agni}
Sharad Agnihotri.
\newblock Quantum cohomology and the {V}erlinde algebra.
\newblock {\em Ph. D. Thesis, University of Oxford, 1995}, (7):1634--1680,
  1995.

\bibitem[Bea96]{Beauville}
Arnaud Beauville.
\newblock Conformal blocks, fusion rules and the {V}erlinde formula.
\newblock In {\em Proceedings of the {H}irzebruch 65 {C}onference on
  {A}lgebraic {G}eometry ({R}amat {G}an, 1993)}, volume~9 of {\em Israel Math.
  Conf. Proc.}, pages 75--96. Bar-Ilan Univ., Ramat Gan, 1996.

\bibitem[Bel04]{PB}
Prakash Belkale.
\newblock Invariant theory of {${\rm GL}(n)$} and intersection theory of
  {G}rassmannians.
\newblock {\em Int. Math. Res. Not.}, (69):3709--3721, 2004.
\newblock URL: \url{http://dx.doi.org/10.1155/S107379280414155X}.

\bibitem[Bel08a]{b4}
Prakash Belkale.
\newblock Quantum generalization of the {H}orn conjecture.
\newblock {\em J. Amer. Math. Soc.}, 21(2):365--408, 2008.
\newblock URL: \url{http://dx.doi.org/10.1090/S0894-0347-07-00584-X}.

\bibitem[Bel08b]{Belkale1}
Prakash Belkale.
\newblock The strange duality conjecture for generic curves.
\newblock {\em J. Amer. Math. Soc.}, 21(1):235--258 (electronic), 2008.
\newblock URL: \url{http://dx.doi.org/10.1090/S0894-0347-07-00569-3}.

\bibitem[Bel10]{belkaleicm}
Prakash Belkale.
\newblock The tangent space to an enumerative problem.
\newblock In {\em Proceedings of the {I}nternational {C}ongress of
  {M}athematicians. {V}olume {II}}, pages 405--426. Hindustan Book Agency, New
  Delhi, 2010.

\bibitem[Ber97]{Bertram}
Aaron Bertram.
\newblock Quantum {S}chubert calculus.
\newblock {\em Adv. Math.}, 128(2):289--305, 1997.
\newblock URL: \url{http://dx.doi.org/10.1006/aima.1997.1627}.

\bibitem[BGM13]{BGMB}
Prakash {Belkale}, Angela {Gibney}, and S.~{Mukhopadhyay}.
\newblock {Nonvanishing of conformal blocks divisors}.
\newblock {\em https://sites.google.com/site/angelagibney/}, August 2013.

\bibitem[BK10]{bk2}
Prakash Belkale and Shrawan Kumar.
\newblock Eigencone, saturation and {H}orn problems for symplectic and odd
  orthogonal groups.
\newblock {\em J. Algebraic Geom.}, 19(2):199--242, 2010.
\newblock URL: \url{http://dx.doi.org/10.1090/S1056-3911-09-00517-7}.

\bibitem[CT13]{MDSBS}
A.-M. {Castravet} and J.~{Tevelev}.
\newblock {$\overline{\rm M}_{0,n}$ is not a Mori Dream Space}.
\newblock {\em ArXiv e-prints}, November 2013.
\newblock \href {http://arxiv.org/abs/1311.7673} {\path{arXiv:1311.7673}}.

\bibitem[Fak12]{Fakh}
Najmuddin Fakhruddin.
\newblock Chern classes of conformal blocks.
\newblock In {\em Compact moduli spaces and vector bundles}, volume 564 of {\em
  Contemp. Math.}, pages 145--176. Amer. Math. Soc., Providence, RI, 2012.
\newblock URL: \url{http://dx.doi.org/10.1090/conm/564/11148}.

\bibitem[Fal93]{Fa}
Gerd Faltings.
\newblock Stable {$G$}-bundles and projective connections.
\newblock {\em J. Algebraic Geom.}, 2(3):507--568, 1993.

\bibitem[FSV94]{FSV1}
Boris Feigin, Vadim Schechtman, and Alexander Varchenko.
\newblock On algebraic equations satisfied by hypergeometric correlators in
  {WZW} models. {I}.
\newblock {\em Comm. Math. Phys.}, 163(1):173--184, 1994.
\newblock URL: \url{http://projecteuclid.org/euclid.cmp/1104270384}.

\bibitem[FSV95]{FSV2}
Boris Feigin, Vadim Schechtman, and Alexander Varchenko.
\newblock On algebraic equations satisfied by hypergeometric correlators in
  {WZW} models. {II}.
\newblock {\em Comm. Math. Phys.}, 170(1):219--247, 1995.
\newblock URL: \url{http://projecteuclid.org/euclid.cmp/1104272957}.

\bibitem[Ful00]{fulton}
William Fulton.
\newblock Eigenvalues, invariant factors, highest weights, and {S}chubert
  calculus.
\newblock {\em Bull. Amer. Math. Soc. (N.S.)}, 37(3):209--249 (electronic),
  2000.
\newblock URL: \url{http://dx.doi.org/10.1090/S0273-0979-00-00865-X}.

\bibitem[GK14]{GoKa}
J.~{Gonz\'alez} and K.~{Karu}.
\newblock {Some non-finitely generated Cox rings}.
\newblock {\em ArXiv e-prints}, July 2014.
\newblock \href {http://arxiv.org/abs/1407.6344} {\path{arXiv:1407.6344}}.

\bibitem[GW86]{GW}
Doron Gepner and Edward Witten.
\newblock String theory on group manifolds.
\newblock {\em Nuclear Phys. B}, 278(3):493--549, 1986.
\newblock URL: \url{http://dx.doi.org/10.1016/0550-3213(86)90051-9}.

\bibitem[Has03]{HassettWeighted}
Brendan Hassett.
\newblock Moduli spaces of weighted pointed stable curves.
\newblock {\em Adv. Math.}, 173(2):316--352, 2003.
\newblock URL: \url{http://dx.doi.org/10.1016/S0001-8708(02)00058-0}.

\bibitem[Kle74]{Kleiman}
Steven~L. Kleiman.
\newblock The transversality of a general translate.
\newblock {\em Compositio Mathematica}, 28(3):287--297, 1974.
\newblock URL: \url{http://eudml.org/doc/89215}.

\bibitem[Kly98]{klyachko}
Alexander~A. Klyachko.
\newblock Stable bundles, representation theory and {H}ermitian operators.
\newblock {\em Selecta Math. (N.S.)}, 4(3):419--445, 1998.
\newblock URL: \url{http://dx.doi.org/10.1007/s000290050037}.

\bibitem[KT99]{knutsontao}
Allen Knutson and Terence Tao.
\newblock The honeycomb model of {${\rm GL}\sb n({\bf C})$} tensor products.
  {I}. {P}roof of the saturation conjecture.
\newblock {\em J. Amer. Math. Soc.}, 12(4):1055--1090, 1999.
\newblock URL: \url{http://dx.doi.org/10.1090/S0894-0347-99-00299-4}.

\bibitem[Les47]{Le}
L{\'e}once Lesieur.
\newblock Les probl\`emes d'intersection sur une vari\'et\'e de {G}rassmann.
\newblock {\em C. R. Acad. Sci. Paris}, 225:916--917, 1947.

\bibitem[MO07]{MO2}
Alina Marian and Dragos Oprea.
\newblock The level-rank duality for non-abelian theta functions.
\newblock {\em Invent. Math.}, 168(2):225--247, 2007.
\newblock URL: \url{http://dx.doi.org/10.1007/s00222-006-0032-z}.

\bibitem[{Muk}13]{SM}
S.~{Mukhopadhyay}.
\newblock {Rank-Level duality and Conformal Block divisors}.
\newblock {\em ArXiv e-prints}, August 2013.
\newblock \href {http://arxiv.org/abs/1308.0854} {\path{arXiv:1308.0854}}.

\bibitem[NT92]{NakTsu}
Tomoki Nakanishi and Akihiro Tsuchiya.
\newblock Level-rank duality of {WZW} models in conformal field theory.
\newblock {\em Comm. Math. Phys.}, 144(2):351--372, 1992.
\newblock URL: \url{http://projecteuclid.org/euclid.cmp/1104249321}.

\bibitem[Oud11]{RO}
R{\'e}my Oudompheng.
\newblock Rank-level duality for conformal blocks of the linear group.
\newblock {\em J. Algebraic Geom.}, 20(3):559--597, 2011.
\newblock URL: \url{http://dx.doi.org/10.1090/S1056-3911-2010-00529-2}.

\bibitem[Pau96]{pauly}
Christian Pauly.
\newblock Espaces de modules de fibr\'es paraboliques et blocs conformes.
\newblock {\em Duke Math. J.}, 84(1):217--235, 1996.
\newblock URL: \url{http://dx.doi.org/10.1215/S0012-7094-96-08408-2}.

\bibitem[Sor96]{Sorger}
Christoph Sorger.
\newblock La formule de {V}erlinde.
\newblock {\em Ast\'erisque}, (237):Exp.\ No.\ 794, 3, 87--114, 1996.
\newblock S{\'e}minaire Bourbaki, Vol. 1994/95.

\bibitem[{Swi}10]{ConfBlocks}
D.~{Swinarski}.
\newblock \texttt{\upshape conformalblocks}: a macaulay2 package for computing
  conformal block divisors.
\newblock 2010.
\newblock Version 1.1, http://www.math.uiuc.edu/Macaulay2/.

\bibitem[Tsu93]{Tsuchimoto}
Yoshifumi. Tsuchimoto.
\newblock On the coordinate-free description of conformal blocks.
\newblock {\em J. Math. Kyoto. Univ.}, 1:29--49, 1993.

\bibitem[TUY89]{TUY}
Akihiro Tsuchiya, Kenji Ueno, and Yasuhiko Yamada.
\newblock Conformal field theory on universal family of stable curves with
  gauge symmetries.
\newblock In {\em Integrable systems in quantum field theory and statistical
  mechanics}, volume~19 of {\em Adv. Stud. Pure Math.}, pages 459--566.
  Academic Press, Boston, MA, 1989.

\bibitem[Wit95]{witten}
Edward Witten.
\newblock The {V}erlinde algebra and the cohomology of the {G}rassmannian.
\newblock In {\em Geometry, topology, \& physics}, Conf. Proc. Lecture Notes
  Geom. Topology, IV, pages 357--422. Int. Press, Cambridge, MA, 1995.

\end{thebibliography}

\end{document}